\documentclass[11pt,reqno]{amsart}

\usepackage{geometry}
\usepackage{amsmath}
\usepackage{amssymb}

\usepackage{graphicx}
\usepackage{tikz-cd}
\graphicspath{ {./Figures/} }

\usepackage{import}
\usepackage{xifthen}
\usepackage{pdfpages}
\usepackage{transparent}

\usepackage{hyperref}
\usepackage{xcolor}
\definecolor{teosgreen}{RGB}{28,151,126}
\definecolor{teossalmon}{RGB}{190,80,80}
\hypersetup{
	colorlinks,
	linkcolor = teossalmon,
	citecolor = teossalmon,
	urlcolor = teossalmon,
	linktoc = page
}

\newcommand{\GG}[1]{}

\newtheorem{theorem}{Theorem}[section]
\newtheorem{lemma}[theorem]{Lemma}
\newtheorem{proposition}[theorem]{Proposition}
\newtheorem{corollary}[theorem]{Corollary}
\newtheorem{conjecture}[theorem]{Conjecture}

\theoremstyle{definition}
	\newtheorem{definition}[theorem]{Definition}
	\newtheorem{remark}[theorem]{Remark}
	\newtheorem{example}[theorem]{Example}

\newcommand{\set}[1]{\left\{ #1 \right\}}
\newcommand{\abs}[1]{\left\lvert #1 \right\rvert}
\newcommand{\norm}[1]{\left\lVert#1\right\rVert}

\newcommand{\Z}{\mathbb{Z}}
\newcommand{\Q}{\mathbb{Q}}
\newcommand{\R}{\mathbb{R}}
\newcommand{\Sp}{\mathbb{S}}
\newcommand{\RP}{\mathbb{R}\mathrm{P}}
\newcommand{\D}{\mathbb{D}}

\DeclareMathOperator{\fillrad}{fillrad}
\DeclareMathOperator{\scal}{scal}

\DeclareMathOperator{\diam}{diam}
\DeclareMathOperator{\inj}{inj}

\DeclareMathOperator{\intt}{int}

\DeclareMathOperator{\II}{II}

\title[Complete 3-manifolds of positive scalar curvature with quadratic decay]{Complete 3-manifolds of positive scalar curvature with quadratic decay}

\author[F. Balacheff]{Florent Balacheff}
\address{Florent Balacheff, Universitat Aut\`onoma de Barcelona and Centre de Recerca Matem\`atica, 08193 Bellaterra, Spain}
\email{florent.balacheff@uab.cat}

\author[T. Gil Moreno de Mora Sard\`a]{Teo Gil Moreno de Mora Sard\`a}
\address{Teo Gil Moreno de Mora Sard\`a, Univ Paris Est Creteil, CNRS, LAMA, F-94010 Creteil, France;
Univ Gustave Eiffel, LAMA, F-77447 Marne-la-Vall\'ee, France; Departament de Matem\`atiques, Universitat Aut\`onoma de Barcelona, Barcelona, Spain}
\email{teo.gil-moreno-de-mora-i-sarda@u-pec.fr}

\author[S. Sabourau]{St\'ephane Sabourau}
\address{St\'ephane Sabourau\\ Univ Paris Est Creteil, CNRS, LAMA, F-94010 Creteil, France;
Univ Gustave Eiffel, LAMA, F-77447 Marne-la-Vall\'ee, France}
\email{stephane.sabourau@u-pec.fr}

\date{\today}

\makeatletter 
\@namedef{subjclassname@2020}{\textup{2020} Mathematics Subject Classification} 
\makeatother
\subjclass[2020]{Primary 53C23; Secondary 53C21}
\keywords{Scalar curvature, quadratic decay, fill radius, $3$-manifold topology.}

\thanks{The first and second authors acknowledge support by the FEDER/AEI/MICINN grant PID2021-125625NB-I00 and the AGAUR grant 2021-SGR-01015. The first author acknowledges support by the AEI/MICINN Mar\'ia de Maeztu grant CEX2020-001084-M.  The second and third authors acknowledge support by the project Min-Max (ANR-19-CE40-0014).}

\begin{document}

\begin{abstract}
	We prove that if an orientable 3-manifold $M$ admits a complete Riemannian metric whose scalar curvature is positive and has a subquadratic decay at infinity, then it decomposes as a (possibly infinite) connected sum of spherical manifolds and $\Sp^2 \times \Sp^1$ summands. This generalises a theorem of Gromov and Wang by using a different, more topological, approach. As a result, the manifold $M$ carries a complete Riemannian metric of uniformly positive scalar curvature, which partially answers a conjecture of Gromov. More generally, the topological decomposition holds without any scalar curvature assumption under a weaker condition on the filling discs of closed curves in the universal cover based on the notion of fill radius. Moreover, the decay rate of the scalar curvature is optimal in this decomposition theorem. Indeed, the manifold $\R^2 \times \Sp ^1$ supports a complete metric of positive scalar curvature with exactly quadratic decay, but does not admit a decomposition as a connected sum.
\end{abstract}

\maketitle

\section{Introduction}

The scalar curvature $\scal : M \rightarrow \R$  of a Riemannian $n$-manifold is a central invariant in Riemannian geometry, defined as the average of the sectional curvatures up to a multiplicative constant. A main objective in Riemannian geometry is to extract topological or geometrical information from conditions on the scalar curvature of a manifold. In the study of three-dimensional manifolds, a fundamental question consists in understanding the topological structure of 3-manifolds that admit a Riemannian metric of positive scalar curvature. In his Problem Section \cite{Yau_1982}, Yau asked for a classification of such manifolds.

The closed case was addressed by Schoen--Yau \cite{Schoen_Yau_1979_a,Schoen_Yau_1979_b} using minimal surfaces and in parallel by Gromov--Lawson \cite{Gromov_Lawson_1980a, Gromov_Lawson_1980b,Gromov_Lawson_1983} using both minimal surfaces and the Dirac operator method, and finally concluded in the light of Perelman's work \cite{Perelman_2002,Perelman_2003_a,Perelman_2003_b}. They proved that a closed orientable 3-manifold which admits a Riemannian metric with positive scalar curvature decomposes as a connected sum of spherical manifolds and $\Sp^2 \times \Sp^1$ summands. Here, a \emph{spherical manifold} is a manifold $\Sp^3 / \Gamma$ obtained as the quotient of the 3-sphere by a subgroup $\Gamma < O(4)$ of isometries acting freely on $\Sp^3$. In both cases, their proof relies on the classical Kneser--Milnor prime decomposition theorem \cite{Kneser_1929,Milnor_1962} (and more generally on the resolution of the Poincar\'e conjecture).

The first problem one encounters when considering the non-compact case is that the Kneser--Milnor prime decomposition does not hold in general. Some counterexamples have been constructed \cite{Scott_1977,Scott_Tucker_1989,Maillot_2008}, even among infinite connected sums. However, a similar decomposition theorem has recently been proved for open manifolds admitting complete Riemannian metrics of uniformly positive scalar curvature (see Definition \ref{de:decay_infinity}): more specifically, for manifolds with finitely generated fundamental group using $K$-theory methods \cite{Chang_Weinberger_Yu_2010}, for manifolds with bounded geometry using Ricci flow techniques \cite{Bessieres_Besson_Maillot_2011}, and finally in the general case independently by Gromov \cite{Gromov_2023} and Wang \cite{Wang_2023}, using $\mu$-bubble theory.

\begin{theorem}[{\cite{Gromov_2023,Wang_2023}}] \label{th:Gromov_Wang}
	Let $M$ be a complete orientable Riemannian 3-manifold with uniformly positive scalar curvature. Then $M$ decomposes as a possibly infinite connected sum of spherical manifolds and $\Sp^2 \times \Sp^1$.
\end{theorem}

In this paper, we will consider 3-manifolds admitting a complete Riemannian metric of positive scalar curvature with at most a quadratic decay at infinity.

\begin{definition} \label{de:decay_infinity}
	Let $M$ be a complete Riemannian $n$-dimensional manifold. Fix a basepoint $x \in M$, and denote by $r_x(y) = d(x,y)$ the distance function to $x$.
	
	\begin{enumerate}
		\item The scalar curvature of $M$ is \emph{uniformly positive} if there exists a constant $s_0 > 0$ such that $\scal \geq s_0 > 0$.
		\item The scalar curvature of $M$ has a \emph{decay at infinity of rate $\alpha \geq 0$ and constant $C > 0$} if there exists a constant $R_0 > 0$ such that for every $y \in M$ with $r_x(y) \geq R_0$,
	\begin{equation*}
		\scal(y) > \frac{C}{r_x(y)^\alpha}.
	\end{equation*}
	Notice that a change of basepoint would only modify the constant $R_0$ and would leave the constant $C$ unaltered. Hence, the basepoint $x \in M$ can be taken arbitrarily.
	\item The scalar curvature of $M$ has a \emph{subquadratic decay at infinity} if it decays at infinity at rate~$\alpha$ and constant $C$, for some $\alpha < 2$ and $C > 0$ .
	\item The scalar curvature of $M$ has \emph{at most $C$-quadratic decay at infinity} if it has a decay at infinity of rate $\alpha = 2$ and constant $C$, for some $C > 0$.
	\end{enumerate}
\end{definition}
With these definitions, if the scalar curvature of $M$ is uniformly positive with $\scal \geq s_0$, it decays at infinity at rate $\alpha$ and constant $C$, for any $\alpha \geq 0$ and $0 < C \leq s_0$. Similarly, if the scalar curvature of $M$ has a subquadratic decay at infinity, it has at most $C$-quadratic decay at infinity for any $C>0$. \medskip

Our main theorem extends the decomposition in Theorem \ref{th:Gromov_Wang} to complete Riemannian 3-manifolds of positive scalar curvature with at most a quadratic decay at infinity for some constant $C > 64 \pi^2$.

\begin{theorem} \label{th:main}
	Let $M$ be a complete orientable Riemannian 3-manifold. Suppose that $M$ has positive scalar curvature with at most $C$-quadratic decay at infinity for some $C > 64 \pi^2$.
	Then $M$ decomposes as a possibly infinite connected sum of spherical manifolds and $\Sp^2 \times \Sp^1$ summands.
\end{theorem}

The notion of infinite connected sum (modelled on a locally finite graph) will be reviewed and discussed in Section \ref{se:prime_decomposition}. Notice that in particular, if the scalar curvature is uniformly positive, then it has at most $C$-quadratic decay at infinity, for any $C > 64 \pi^2$. Thus, we recover Gromov and Wang's topological classification result under a weaker assumption with an approach focusing on the fill radius notion (see Theorem \ref{th:main_general} for a more general curvature-free result).\medskip

One may wonder whether the conclusion of Theorem \ref{th:main} holds under a weaker decay rate (that is, for $\alpha > 2$). The example of the manifold $\R^2 \times \Sp^1$ \cite[Section 3.10.2]{Gromov_2023} shows this is impossible. Indeed, the manifold $\R^2 \times \Sp^1$ admits a complete metric of positive scalar curvature decaying $\frac{1}{2}$-quadratically at infinity, but it does not decompose as an infinite connected sum of spherical manifolds and $\Sp^2 \times \Sp^1$; see Section \ref{se:example_R2xS1}. Therefore, the decay rate in Theorem \ref{th:main} is optimal.

As for the optimal value of the decay constant $C$ under which the conclusion of Theorem \ref{th:main} holds, this last example shows that we cannot hope for more than $C > \frac{1}{2}$ (while our result holds for $C > 64\pi^2$).

More generally, Gromov conjectured the following \cite[Section 3.6.1]{Gromov_2023}.
\begin{conjecture}[{Critical Rate of Decay Conjecture \cite{Gromov_2023}}] \label{co:critical_rate}
	There exists a dimensional constant $C_n > 0$ such that the following holds. Let $M$ be an orientable $n$-manifold that admits a complete Riemannian metric of positive scalar curvature.
	\begin{enumerate}
		\item For every $C < C_n$, there exists a complete Riemannian metric on $M$ of positive scalar curvature with at most $C$-quadratic decay at infinity.
		\item If $M$ admits a complete Riemannian metric with positive scalar curvature with $C$-quadratic decay at infinity for $C > C_n$, then $M$ admits a complete Riemannian metric with uniformly positive scalar curvature. \label{it:co:critical_rate}
	\end{enumerate}
\end{conjecture}

The following rigidity result, which addresses the case \eqref{it:co:critical_rate} of Conjecture \ref{co:critical_rate}, is a direct consequence of Theorem \ref{th:main} and an adaptation of Gromov--Lawson's Surgery Theorem \cite[Theorem A]{Gromov_Lawson_1980b}; see Section \ref{se:surgery}.

\begin{corollary} \label{co:surgery_uniform_scal}
	Let $M$ be an orientable 3-manifold. If $M$ admits a complete Riemannian metric of positive scalar curvature with at most $C$-quadratic decay at infinity for some $C > 64 \pi^2$, it also admits a complete Riemannian metric with uniformly positive scalar curvature.
\end{corollary}

Actually, we will deduce Theorem \ref{th:main} from a more general statement, which involves the following notion of fill radius, introduced in \cite[Section 10]{Gromov_Lawson_1983} and \cite{Schoen_Yau_1983}, (related to the more general notion of filling radius presented in \cite{Gromov_1983}).

\begin{definition} \label{de:metric_neighbourhood}
	Let $M$ be a Riemannian $n$-manifold with possibly nonempty boundary. Given a subset $Z \subset M$, denote by
	\begin{equation} \label{eq:metric_neighbourhood}
		U(Z,R) := \set{x \in M \mid d(x,Z) \leq R}
	\end{equation}
	the closed $R$-neighbourhood of $Z$ in $M$.
	
	The \emph{fill radius} of a contractible closed curve $\gamma$ in $M$ is defined as
	\begin{equation*}
		\fillrad{(\gamma)} := \sup\set{R \geq 0 \mid d(\gamma, \partial M) > R \text{ and } [\gamma] \neq 0 \in \pi_1(U(\gamma , R))}.
	\end{equation*}
	Define also
	\begin{equation*}
		\fillrad{(M)} := \sup\set{\fillrad{(\gamma)} \mid \gamma \text{ contractible closed curve of } M}.
	\end{equation*}
\end{definition}

For instance, the standard sphere~$\Sp^2$ and the standard Riemannian cylinder $\Sp^2 \times \R$ satisfy $\fillrad(\Sp^2) = \fillrad{(\Sp^2 \times \R)} = \frac{\pi}{2}$ as observed in \cite[Remark 1]{Schoen_Yau_1983}. 
If a manifold $M$ has bounded diameter, then $\fillrad(M) \leq \diam(M)$. In particular, by the Bonnet-Myers theorem, a manifold of uniformly positive Ricci curvature must have bounded fill radius. In \cite{Ramachandran_Wolfson_2010}, the authors conjectured that some weaker conditions of positive curvature (namely, uniformly 2-positive Ricci curvature and uniformly positive isotropic curvature) also imply a bound on the fill radius. For 3-manifolds of uniformly positive scalar curvature, the following bound on the fill radius was established in \cite{Gromov_Lawson_1983,Schoen_Yau_1983}; see also \cite[Theorem 9.3.1]{Wolfson_2012} for another exposition.

\begin{theorem}[{\cite[Theorem 10.7]{Gromov_Lawson_1983}, \cite{Schoen_Yau_1983}}] \label{th:Gromov_Lawson}
	Let $M$ be a complete Riemannian 3-manifold with bounded geometry, possibly with nonempty boundary. If $\scal \geq s_0 > 0$, then
	\begin{equation} \label{eq:Gromov_Lawson_fillrad_estimate}
		\fillrad(M) \leq \frac{2\pi}{\sqrt{s_0}}.
	\end{equation}
\end{theorem}

The value of the constant $2\pi$ was obtained by Gromov--Lawson \cite{Gromov_Lawson_1983}. Using a different argument, Schoen--Yau \cite{Schoen_Yau_1983} derived Theorem \ref{th:Gromov_Lawson} for the sharper value of the constant $\sqrt{\frac{8}{3}}\pi$, and recently Hu--Xu--Zhang in  \cite{Hu_Xu_Zhang_2024} showed that the latter is optimal. However, we will use Gromov--Lawson's constant, since their proof is more suitable to our case; see Proposition \ref{pr:corollary_GL}.

Theorem \ref{th:Gromov_Lawson} implies that when $M$ is a complete orientable 3-manifold with bounded geometry and uniformly positive scalar curvature $\scal \geq s_0 > 0$, then the universal Riemannian cover $\tilde{M}$ of~$M$ satisfies $\fillrad{(\tilde{M})} \leq 2\pi / \sqrt{s_0}$. Therefore, an upper bound on the fill radius of the universal cover provides a generalisation of the notion of uniformly positive scalar curvature.\medskip

If a complete orientable 3-manifold $M$ has positive scalar curvature decaying at infinity, then the fill radius is not necessarily bounded in general. Still, if the decay is not too pronounced, one can control the growth of the fill radius of the lifts to the universal cover of the closed curves contractible in $M$. This property will serve as a generalisation of the notion of positive scalar curvature with at most $C$-quadratic decay at infinity for some $C > 64 \pi^2$; see Proposition \ref{pr:generalised_condition}. Notice that if $M$ is a compact manifold, the fill radius of contractible curves of $M$ is uniformly bounded by $\diam{(M)}$, which is why we consider the lifts of such curves to the universal cover.\medskip

\begin{definition}
	Let $M$ be a complete Riemannian manifold, and denote by $\tilde{M}$ its universal Riemannian cover. Fix a basepoint $x \in M$. Denote by $B(x,R)$ the closed metric ball of radius $R$ centered at $x$.
	\begin{enumerate}
		\item The fill radius of $\tilde{M}$ \emph{grows at infinity at rate} $\beta \geq 0$ \emph{and with constant} $c > 0$ if there is a constant $R'_0 \geq 0$ such that if $R \geq R'_0$, then for every closed curve $\gamma$ lying in $B(x,R)$ and contractible in $M$, any of its lifts $\tilde{\gamma}$ to $\tilde{M}$ satisfies
	\begin{equation*}
		\fillrad{(\tilde{\gamma})} < c R^\beta.
	\end{equation*}
	Again, notice that the basepoint $x \in M$ may be taken arbitrarily since a change of basepoint does not affect the rate $\beta$, nor the value of the constant $c$.
	\item The fill radius of $\tilde{M}$ has \emph{sublinear growth at infinity} if it grows at infinity at rate $\beta$ and constant $c$, for some $\beta < 1$ and $c > 0$.
	\item The fill radius of $\tilde{M}$ has \emph{at most $c$-linear growth at infinity} if it grows a rate $\beta = 1$ and constant $c$, for some $c > 0$.
	\end{enumerate}
\end{definition}

For instance, the fill radius of the Euclidean 3-dimensional space $\R^3$ grows at infinity at rate~$\beta = 1$ for every constant $c > 1$. In the same direction, we have the following result.

\begin{proposition} \label{pr:generalised_condition}
	Let $M$ be an orientable complete Riemannian 3-manifold, and denote by $\tilde{M}$ its universal Riemannian cover. Suppose that $M$ has positive scalar curvature with at most $C$-quadratic decay at infinity for some $C > 64 \pi^2$. Then, the fill radius of $\tilde{M}$ has at most $c$-linear growth at infinity for some $c < \frac{1}{3}$.
\end{proposition}

We will prove the topological decomposition of Theorem \ref{th:main} by replacing the scalar curvature assumption with this weaker condition about the filling discs of the lifts of contractible closed curves, namely that the fill radius of $\tilde{M}$ has at most $c$-linear growth at infinity with $c < \frac{1}{3}$.

\begin{theorem} \label{th:main_general}
	Let $M$ be an orientable complete Riemannian 3-manifold, and denote by $\tilde{M}$ its universal Riemannian cover. Suppose that the fill radius of $\tilde{M}$ has at most $c$-linear growth at infinity for some $c < \frac{1}{3}$. Then $M$ decomposes as a possibly infinite connected sum of spherical manifolds and $\Sp^2 \times \Sp^1$.
\end{theorem}

This result yields the same topological decomposition as Gromov--Wang's theorem under a weaker, more robust, assumption. In particular, it applies to metrics of positive scalar curvature with at most $C$-quadratic decay at infinity for some $C > 64 \pi^2$, and not just of uniformly positive scalar curvature. More generally, Theorem \ref{th:main_general} does not require any curvature assumption and relies on $\mathcal{C}^0$ topological arguments, rather than on $\mathcal{C}^2$ analytical ones. In particular, the proof of Theorem \ref{th:main_general} relies neither on the $\mu$-bubble theory, nor on the minimal surface approach. Interestingly, and somewhat surprisingly, this general approach leads to an optimal statement in the decay rate at infinity despite the lack of analytical tools.\medskip

A direct consequence of Theorem \ref{th:main_general} is that complete Riemannian 3-manifolds such that the fill radius of their universal cover has at most $c$-linear growth at infinity for some $c < \frac{1}{3}$ are not aspherical. This extends a theorem of Gromov and Lawson \cite[Theorem E]{Gromov_Lawson_1983} asserting that closed aspherical 3-manifolds do not admit Riemannian metrics with positive scalar curvature. Actually, we prove the following more general statement: closed $\Q$-essential manifolds of dimension $n \geq 2$ do not support Riemannian metrics such that their Riemannian universal cover has finite fill radius; see Proposition \ref{pr:Qessential}. This results answers a weaker version of a conjecture by Gromov in \cite[Section 4, Conjecture 12]{Gromov_2017} and \cite[Section 3.2]{Gromov_2023}; see Conjecture \ref{conj:Gromov_Q_essential}.\\

The paper is structured as follows. Consider an orientable complete Riemannian 3-manifold $M$ with positive scalar curvature of at most $C$-quadratic decay at infinity for some $C > 64 \pi^2$. In Section \ref{se:scal_fillrad}, we prove Proposition \ref{pr:generalised_condition} by showing that the universal cover $\tilde{M}$ has fill radius with at most $c$-linear growth at infinity for some $c < \frac{1}{3}$. In Section \ref{se:simply_connected_infinity}, we prove that $M$ is simply connected at infinity on contractible curves. In Section \ref{se:localised_compression}, we obtain an exhaustion by compact domains for~3-manifolds that are simply connected at infinity on contractible curves. In Section \ref{se:number_ends}, we introduce the notion of ends of a group and we prove that complete Riemannian manifolds whose universal cover has fill radius with $c$-linear growth at infinity for some $c < \frac{1}{3}$ have virtually free fundamental group. This extends a theorem of Ramachandran--Wolfson \cite{Ramachandran_Wolfson_2010} who proved the same result for closed manifolds whose universal cover has bounded fill radius. In Section \ref{se:aspherical_summands}, we show that such manifolds cannot decompose as a connected sum involving an aspherical summand. In Section \ref{se:prime_decomposition}, we prove Theorem \ref{th:main_general}, and we show its optimality in Section \ref{se:example_R2xS1}. Finally, in Section \ref{se:surgery} we prove Corollary \ref{co:surgery_uniform_scal}. \\

\noindent\textit{Acknowledgement}. The authors would like to thank the referee for their detailed and useful comments.

\section{Scalar curvature and fill radius} \label{se:scal_fillrad}

Let $M$ be a complete Riemannian 3-manifold. In coherence with the definition of the fill radius, see Definition~\ref{de:metric_neighbourhood}, we introduce the following notation.
\begin{definition} \label{de:fillrad_relative}
	Let $K$ be a subset of a complete Riemannian 3-manifold $M$. Let $\gamma$ be a closed curve lying in $K$ and contractible in $M$. Define
	\begin{equation*}
		\fillrad{(\gamma \subset K)} := \sup\set{R \geq 0 \mid d(\gamma, \partial K) > R \text{ and } [\gamma] \neq 0 \in \pi_1(U(\gamma , R))}.
	\end{equation*}
	Recall that $U(\gamma,R)$ is the closed $R$-neighbourhood of $\gamma$ in $M$.
\end{definition}

With this notation, one can adapt the proof of Theorem \ref{th:Gromov_Lawson} in \cite[Proof of Theorem 10.7]{Gromov_Lawson_1983} to prove the following result.

\begin{proposition} \label{pr:corollary_GL}
	Let $M$ be a complete Riemannian 3-manifold with $\scal > 0$. Let~$K \subset M$ be a compact subset. Let $s_K > 0$ be a constant such that $\scal(x) \geq s_K > 0$ for every $x \in K$. Then, any closed curve $\gamma$ lying in $K$ and contractible in $M$ satisfies:
	\begin{equation*}
		\fillrad{(\gamma \subset K)} \leq \frac{2\pi}{\sqrt{s_K}}.
	\end{equation*}
\end{proposition}

\begin{proof}
	We argue by contradiction following \cite[Proof of Theorem 10.7]{Gromov_Lawson_1983}. The only difference here is that we will need to carefully apply the Stability Inequality \cite[Inequality 10.17]{Gromov_Lawson_1983} to functions with support in $K$. We also give more details of some parts of the argument.
	
	Let $\gamma$ be a closed curve lying in $K$ and contractible in $M$. Let $\rho > \pi/\sqrt{s_K}$. Suppose that $d(\gamma, \partial K) > 2\rho$ and that $\gamma$ does not bound a disc in $U(\gamma,2\rho)$ (in the proof, all the discs are non-necessarily embedded topological discs). Consider $B \subset M$ a closed metric ball whose interior contains the domain $K$ as well as a disc $D \subset M$ bounded by $\gamma$. After slightly deforming $B$, we may assume that the boundary $\partial B$ of $B$ is a smooth surface. Now, we modify the Riemannian metric in a tubular neighbourhood of $\partial B$ away from $K$ to make it isometric to the Riemannian product~	$\partial B \times [0,1]$. With the modified metric, the subset $B$ is a manifold with (mean) convex boundary. By~\cite{Meeks_Yau_1982}, there exists a disc $\Sigma \subset B$  bounded by $\gamma$ of least area. Observe that $\Sigma$ is a stable minimal disc which is not contained in $U(\gamma,2\rho)$.
	
	Now, consider the level set $\gamma_\rho := \set{x \in \Sigma \mid d_\Sigma(\gamma,x) = \rho}$ which, after considering a smooth aproximation of the function $d_\Sigma(\gamma,\cdot)$, consists of a disjoint collection of closed curves. Since $\gamma$ does not bound a disc in $U(\gamma,2\rho)$, there is at least one connected component $\sigma$ of $\gamma_\rho$ which does not bound a disc in $U(\gamma,2\rho)$. Let $\Omega \subset \Sigma$ be a small neighbourhood of $\sigma$. For every $s \in[0, \rho]$, denote by
	\begin{equation*}
		\Omega(s) := \set{x \in \Sigma \mid d_\Sigma(x,\Omega) \leq s}
	\end{equation*}
	the closed $s$-neighbourhood of $\Omega$ in $\Sigma$; see Figure \ref{fi:stability}.
	\begin{figure}[ht]
		\centering
    \def\svgwidth{0.6\columnwidth}
    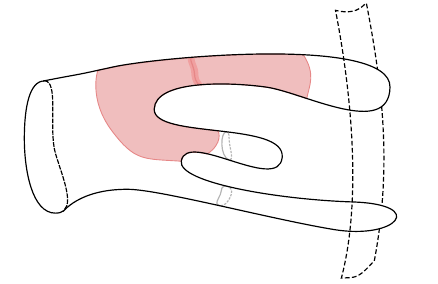

		\caption{The neighbourhood $\Omega(s)$ in the stable minimal disc $\Sigma$.}
		\label{fi:stability}
	\end{figure}
	
	Following \cite[Proof of Theorem 10.2]{Gromov_Lawson_1983}, by analytic approximation of the domain $\Omega$ and the distance function $d_\Sigma(\cdot, \Omega)$, we may assume that the level set $\partial \Omega(s)$ is piecewise smooth, for every~$s \in [0,\rho]$. Notice that, as $d(\gamma, \partial K) > 2\rho$, the set $\Omega(\rho)$ is contained in the compact $K$ and its interior does not intersect the curve $\gamma$.
	
	Using Schoen-Yau's rearrangement \cite{Schoen_Yau_1979_b}, the stability of the minimal surface $\Sigma$ implies
	\begin{equation} \label{eq:stability}
		\int_\Sigma \left(  \abs{\nabla f}^2+ \kappa f^2 -\frac{1}{2} \left(\scal +\norm{\II}^2\right) f^2 \right) dA \geq 0
	\end{equation}
	for any function $f \in \mathcal{C}_c^\infty(\Sigma)$, where $\kappa$ and $\II$ denote the Gauss curvature and the second fundamental form of $\Sigma$.
	Now consider the function $f: \Sigma \rightarrow \R$ defined by
	\begin{equation*}
		f(x) =\begin{cases}
					\cos{\left(\frac{\pi}{2\rho} d_\Sigma (x,\Omega)\right)} & \text{if }x \in \Omega(\rho) \\
					0 & \text{if } x \notin \Omega(\rho) \\
				\end{cases}.
	\end{equation*}
	Since $f$ is supported on $\Omega(\rho)$, which is contained in $K$, the Stability Inequality~\eqref{eq:stability} implies
	\begin{equation*}
		\int_{\Omega(\rho)} \left(  \abs{\nabla f}^2+ \kappa f^2 -\frac{s_K}{2} f^2 \right) dA \geq 0.
	\end{equation*}
	This is the equivalent of the inequality (10.17) in \cite{Gromov_Lawson_1983}, except that now the test function $f$ has support in $K$. Arguing exactly as in \cite[Proof of Theorem 10.2]{Gromov_Lawson_1983}, we conclude that $\rho < \pi / \sqrt{s_K}$, which is a contradiction.
\end{proof}

From Proposition \ref{pr:corollary_GL}, one can derive Proposition \ref{pr:generalised_condition}. More generally, we have the following result.

\begin{proposition} \label{pr:fillrad_estimate}
	Let $M$ be a complete Riemannian 3-manifold, and denote by $\tilde{M}$ its universal Riemannian cover. Suppose that $M$ has positive scalar curvature with a decay at infinity of rate $\alpha \geq 0$ and constant $C > 0$.
	\begin{enumerate}
		\item If the decay is subquadratic (that is, $\alpha \in[0,2)$), then the fill radius of $\tilde{M}$ has sublinear growth at infinity of rate $\beta = \alpha/2 \in [0,1)$.
		\item If the decay is at most $C$-quadratic with $C > 4\pi^2$, then the fill radius of $\tilde{M}$ has at most linear growth at infinity with constant
		\begin{equation*}
			c = \frac{2\pi}{\sqrt{C}-2\pi}.
		\end{equation*}
	\end{enumerate}
	In particular, if the scalar curvature has at most $C$-quadratic decay at infinity with $C > 64 \pi^2$, then the fill radius of $\tilde{M}$ has at most $c$-linear growth at infinity with $c < \frac{1}{3}$.
\end{proposition}

\begin{proof}
	Fix a point $x \in M$. Consider $R_0 > 0$ such that every point $y \in M$ with $r_x(y) \geq R_0$ satisfies
	\begin{equation*}
		\scal(y) > \frac{C}{r_x(y)^\alpha},
	\end{equation*}
	and define $s_0 := \min_{B(x,R_0)} \scal$. Fix $\mu > 1$. Let $R \geq R'_0$, where
	\begin{equation*}
		R'_0 =\begin{cases}
					\max\set{R_0, \left(\frac{C}{s_0}\right)^{\frac{1}{\alpha}}, \left(\frac{2\pi}{\sqrt{C}}\right)^{\frac{2}{2-\alpha}}\left(\frac{\mu^\alpha}{(\mu-1)^2}\right)^{\frac{1}{2-\alpha}}} & \text{if } \alpha \in [0,2) \\
					\max\set{R_0 , \sqrt{\frac{C}{s_0}}} & \text{if } \alpha = 2 \\
				\end{cases}.
	\end{equation*}
	Note that $R \geq \max\set{R_0 , (\frac{C}{s_0})^\frac{1}{\alpha}}$ in both cases.
	
	Let $\gamma$ be a closed curve lying in the closed metric ball $B(x,R)$ and contractible in $M$. Take any lift $\tilde{x}$ of $x$ in the universal cover $\tilde{M}$, and lift $\gamma$ to a closed curve $\tilde{\gamma}$ lying in the set~$p^{-1}(B(x,R)) = U(\pi_1 (M) \cdot \tilde{x}, R)$ formed by the points of $\tilde{M}$ at distance at most $R$ from a point in the orbit of $\tilde{x}$ by the action of the fundamental group of $M$; see Definition \ref{de:metric_neighbourhood} (\ref{eq:metric_neighbourhood}).
	
	Consider the $\mu R$-neighbourhood $K = U(\pi_1 (M) \cdot \tilde{x}, \mu R)$ of the orbit $\pi_1 (M) \cdot \tilde{x}$. Since~$R \geq (\frac{C}{s_0})^\frac{1}{\alpha}$, we have
	\begin{equation*}
		\min_{K} \scal > \frac{C}{(\mu R)^\alpha}.
	\end{equation*}
	By Proposition \ref{pr:corollary_GL}, the closed curve $\tilde{\gamma}$ has
	\begin{equation*}
		\fillrad{(\tilde{\gamma} \subset K)} < \frac{2\pi}{\sqrt{C}} (\mu R)^{\frac{\alpha}{2}}.
	\end{equation*}
	Notice that
	\begin{equation*}
		d(\tilde{\gamma}, \partial K) \geq \mu R - R.
	\end{equation*}
	Hence, the curve $\tilde{\gamma}$ bounds a disc in its $\frac{2\pi}{\sqrt{C}} (\mu R)^{\frac{\alpha}{2}}$-neighbourhood if
	\begin{equation} \label{eq:distance_inequality}
		\mu R - R \geq \frac{2\pi}{\sqrt{C}} (\mu R)^{\frac{\alpha}{2}}.
	\end{equation}
	
	\begin{enumerate}
		\item Let us consider first the subquadratic case, that is, suppose $\alpha \in [0,2)$. The inequality
		\begin{equation*}
			R \geq \left(\frac{2\pi}{\sqrt{C}}\right)^{\frac{2}{2-\alpha}}\left(\frac{\mu^\alpha}{(\mu-1)^2}\right)^{\frac{1}{2-\alpha}}
		\end{equation*}
		is equivalent to the inequality (\ref{eq:distance_inequality}). Therefore,
		\begin{equation*}
			\fillrad{(\tilde{\gamma})} < \frac{2\pi}{\sqrt{C}} \mu^{\frac{\alpha}{2}} R^{\frac{\alpha}{2}}.
		\end{equation*}
		That is, the fill radius of $\tilde{M}$ has sublinear growth of rate $\beta = \alpha / 2$ and constant $c = 2\pi \mu^{\frac{\alpha}{2}} / \sqrt{C}$. Notice that this estimate holds for any $\mu > 1$. In particular, one can let $\mu$ go to 1 in order to make $c$ as close to $2\pi / \sqrt{C}$ as desired, in which case $R'_0$ diverges to infinity.
		
		\item Now, suppose that $\alpha = 2$ and $C > 4\pi^2$. In this case, the inequality~(\ref{eq:distance_inequality}) is equivalent to
		\begin{equation*}
			\mu \geq \frac{\sqrt{C}}{\sqrt{C}-2\pi}.
		\end{equation*}
		Hence, for any $\mu \geq \frac{\sqrt{C}}{\sqrt{C}-2\pi}$, we have
		\begin{equation*}
			\fillrad{(\tilde{\gamma})} < \frac{2\pi}{\sqrt{C}} \mu R.
		\end{equation*}
		In particular,
		\begin{equation*}
			\fillrad{(\tilde{\gamma})} < \frac{2\pi}{\sqrt{C}-2\pi} R.
		\end{equation*}
		That is, the fill radius of $\tilde{M}$ has at most $c$-linear growth, with $c = \frac{2\pi}{\sqrt{C}-2\pi}$. Notice that $C > 64 \pi^2$ if and only if $c < \frac{1}{3}$.
	\end{enumerate}
	\end{proof}

\section{Fill radius and simply connectedness at infinity on contractible curves} \label{se:simply_connected_infinity}

Recall the well-known notion of manifold simply connected at infinity (see \cite[Section 16]{Geoghegan_2008} for a detailed discussion on simply connectedness at infinity and related notions).
	\begin{definition}
		A manifold $M$ is \emph{simply connected at infinity} if for any compact subset $B \subset M$, there is a compact subset $K \subset M$ containing $B$ such that the morphism
		\begin{equation*}
			\pi_1(M-K) \rightarrow \pi_1(M-B)
		\end{equation*}
		induced by the inclusion is trivial.
	\end{definition}
	Intuitively, a manifold is simply connected at infinity if for any compact subset $B$, closed curves sufficiently far from $B$ can be contracted to a point avoiding $B$. Notice that simply connected manifolds are not necessarily simply connected at infinity. For instance, the plane in dimension two and the Whitehead manifold are simply connected (even contractible) but they are not simply connected at infinity. If instead we require this condition to hold only for curves that are contractible in $M$, we say that $M$ is simply connected at infinity on contractible curves. More precisely:
	\begin{definition}
		A manifold $M$ is \emph{simply connected at infinity on contractible curves} if for any compact subset $B \subset M$, there is a compact $K \subset M$ containing $B$ such that
		\begin{equation*}
			\ker{(\pi_1(M-K) \rightarrow \pi_1(M))} = \ker{(\pi_1(M-K) \rightarrow \pi_1(M-B))},
		\end{equation*}
		that is, every closed curve $\gamma \subset M - K$ contractible in $M$ is already contractible in $M-B$.
	\end{definition}
	
	\begin{example} \label{ex:simplyconnectedinfinity_contractiblecurves}
		Clearly, simply connectedness at infinity implies simply connectedness at infinity on contractible curves. The converse is not true, as shown by the following example. Consider the 3-manifold $M$ obtained as an infinite connected sum (see Section \ref{se:prime_decomposition}) of manifolds $\Sp^2 \times \Sp^1$ modelled on the halfline $[0,+\infty)$ with vertices at the integer points. The manifold $M$ is simply connected at infinity on contractible curves. However, $M$ is not simply connected at infinity, since the complementary of any compact contains (infinitely many) non-contractible curves.
	\end{example}
	
	In this section, we prove that complete Riemannian manifolds whose universal Riemannian cover has fill radius with at most $c$-linear growth at infinity with $c < 1$ are simply connected at infinity on contractible curves. Notice that this result is valid for manifolds of any dimension.
	
	\begin{proposition} \label{pr:fillrad_simplyconnectedatinfinity}
		Let $M$ be a complete Riemannian manifold, and denote by $\tilde{M}$ its universal Riemannian cover. Suppose the fill radius of $\tilde{M}$ has at most $c$-linear growth at infinity for some $c < 1$. Then the manifold $M$ is simply connected at infinity on contractible curves.
	\end{proposition}
	
	\begin{proof}
		Fix a point $x \in M$. Consider $c < 1$ and $R'_0 > 0$ such that if $\gamma$ is a contractible closed  curve lying in the closed metric ball $B(x,R)$ for $R \geq R'_0$, then any of its lifts $\tilde{\gamma}$ to the universal cover $\tilde{M}$ satisfies
		\begin{equation*}
			\fillrad{(\tilde{\gamma})} < cR.
		\end{equation*}
		Notice that, since the projection from the universal cover to $M$ is distance non-increasing, the curve~$\gamma$ also satisfies $\fillrad{(\gamma)} < c R$.
		
		For any compact $B \subset M$, choose an $r \geq R'_0$ such that $B \subset B(x,r)$. Let $R := \frac{r}{1-c} > r$ and consider the closed metric ball $K = B(x,R)$. Suppose that $\eta \subset M-K$ is a closed curve that is contractible in~$M$. If $\eta$ bounds a disc in $M-K$, there is nothing to prove, so suppose that~$\eta$ bounds an immersed disc $D$ which intersects $K$. That is, there exists an immersion $\rho: \D \rightarrow D \subset M$ from the unit Euclidean disc $\D$ whose restriction to $\partial \D$ coincides with $\eta$. Approximate the distance function~$d(x,\cdot)$ on the manifold $M$ by a smooth function, and denote by $f$ its restriction to the disc~$D$. 
		By Sard's theorem, we can slightly adjust $R$ to a regular value of $f \circ \rho$ so that the preimage $(f \circ \rho)^{-1}(R)$ is a finite collection of disjoint simple closed curves, each of which bounds a topological disc $\D_i \subset \D$. Note that two such discs are either disjoint or one is contained within the other. It follows that the maximal discs $\D_j^+$ in this family are disjoint. Note also that $\rho(\D - \sqcup_j \D_j^+) \subset M-K$. The images~$\gamma_j$ of~$\partial \D_j^+$ under $\rho$ are contractible closed curves in $D$, and hence in $M$. Since each curve $\gamma_j$ lies in the metric ball $B(x,R)$, we have
		\begin{equation*}
			\fillrad{(\gamma_j)} < cR = \frac{c}{1-c} r.
		\end{equation*}
		The distance of each curve $\gamma_j$ to the boundary $\partial B(x,r)$ satisfies
		\begin{equation*}
			d(\gamma_j, \partial B(x,r)) \geq R - r \geq \frac{c}{1-c}r.
		\end{equation*}
		Therefore, each curve $\gamma_j$ bounds a disc inside the complement $M - B(x,r)$. That is, for each closed curve $\gamma_j$, there exists an immersion $\rho_j : \D_j^+ \rightarrow M - B(x,r)$ whose restriction to $\partial \D_j^+$ coincides with~$\gamma_j$. By combining the maps $\rho: \D - \sqcup_j \D_j^+ \rightarrow M-K$ and $\rho_j : \D_j^+  \rightarrow M - B(x,r)$, we obtain a map $\D \rightarrow M - B(x,r)$ whose restriction to $\partial \D$ coincides with $\eta$. Thus, the curve $\eta$ bounds a disc inside $M - B(x,r) \subset M - B$.
	\end{proof}
	
	\begin{remark}
		A result closely related to Proposition \ref{pr:fillrad_simplyconnectedatinfinity} is \cite[Corollary 10.9]{Gromov_Lawson_1983}, where the authors proved that a complete 3-manifold of uniformly positive scalar curvature and finitely generated fundamental group is simply connected at infinity. Note that this result fails without the assumption that the fundamental group is finitely generated; see Example \ref{ex:simplyconnectedinfinity_contractiblecurves} for a counterexample. The hypothesis in Proposition \ref{pr:fillrad_simplyconnectedatinfinity} are more general (in particular, we do not make any assumption on the dimension of $M$, nor on the fundamental group of~$M$), but we derive a weaker result, namely simply connectedness at infinity on contractible curves.
	\end{remark}
	
	\begin{remark}
		By Proposition \ref{pr:fillrad_estimate}, Proposition \ref{pr:fillrad_simplyconnectedatinfinity} applies to complete Riemannian 3-manifolds of positive scalar curvature with at most $C$-quadratic decay at infinity with $C > 16\pi^2$. Although we shall not make use of this fact, one can show directly this consequence by constructing a minimising annulus in a compact domain $K \subset M$ where the scalar curvature is bounded from below by some~$s_K > 0$, without relying on the fill radius, as in the proof of Proposition \ref{pr:corollary_GL}. This alternative approach applies when $C > \pi^2$ and was first used in \cite[Proof of Corollary 10.9]{Gromov_Lawson_1983}. Notice that for such manifolds, the fill radius of their universal cover has at most $c$-linear growth at infinity, for any $c > 0$.
	\end{remark}
	
	The following result is a direct consequence of the application of Proposition \ref{pr:fillrad_simplyconnectedatinfinity} to contractible manifolds.
	
	\begin{corollary}
		Let $M$ be a complete contractible Riemannian manifold. Suppose the fill radius of~$M$ has at most $c$-linear growth at infinity for some $c<1$. Then $M$ is homeomorphic to $\R^n$.
	\end{corollary}
	
	\begin{proof}
		Proposition \ref{pr:fillrad_simplyconnectedatinfinity} together with the fact that the manifold $M$ is contractible imply that $M$ is simply connected at infinity. And the only contractible manifold which is simply connected at infinity is $\R^n$. This result was proven first for manifolds of dimension $n \geq 5$ \cite[Theorem 4]{Stallings_1962}, then for 3-dimensional manifolds \cite{Edwards_1963}, and finally for 4-dimensional manifolds \cite[Corollary 1.2]{Freedman_1982}.
	\end{proof}
	
	In particular, any contractible complete 3-manifold with positive scalar curvature of subquadratic decay at infinity, and more generally of $C$-quadratic decay at infinity for some $C > 16\pi^2$, is homeomorphic to $\R^3$. This result was already obtained by Wang \cite{Wang_2019}. It is an open question whether any contractible complete 3-manifold with positive scalar curvature is homeomorphic to~$\R^3$, despite some recent progress in this direction \cite{Wang_2023_Review,Wang_2024_JDG,Wang_2024_JEMS,Chodosh_Lai_Xu_2025}.

\section{Localised compression and exhaustion by compact domains with incompressible boundaries} \label{se:localised_compression}

Recall the following well-known definition.

\begin{definition}
	Let $\Sigma$ be a closed orientable (not necessarily connected) surface embedded into an orientable 3-manifold $M$. The surface $\Sigma$ is \emph{incompressible} if the morphism $\pi_1(\Sigma,x) \rightarrow \pi_1(M,x)$ induced by inclusion is injective for any basepoint $x \in \Sigma$. Otherwise, the surface $\Sigma$ is \emph{compressible}.
\end{definition}

A compressible surface $\Sigma$ can be compressed into an (embedded) incompressible surface homologous to $\Sigma$ as follows. This procedure is based on the Loop Theorem, proved by Papakyriakopoulos \cite{Papakyriakopoulos_1957}. Here, we shall use the following version of the Loop Theorem \cite[Chapter 4, Theorem 4.2]{Hempel_1976}.

\begin{theorem}[Loop Theorem] \label{th:looptheorem}
	Let $\Sigma$ be a closed orientable surface embedded into an orientable 3-manifold $M$. If, for a certain basepoint $x \in \Sigma$, the homomorphism induced by the inclusion $\pi_1(\Sigma,x) \rightarrow \pi_1(M,x)$ is not injective, then there is a simple closed curve $\gamma$ of $\Sigma$ representing a nontrivial element of $\ker{(\pi_1(\Sigma,x) \rightarrow \pi_1(M,x))}$ and an embedded disc $D \subset M$ such that $\gamma = \partial D = D \cap \Sigma$.
\end{theorem}

The compression of a compressible surface $\Sigma$ is carried out as follows. By the Loop Theorem~\ref{th:looptheorem}, take a simple closed curve $\gamma$ in $\Sigma$ representing a nontrivial element of $\ker{(\pi_1(\Sigma,x) \rightarrow \pi_1(M,x))}$ which bounds an embedded disc $D \subset M$ intersecting $\Sigma$ only along its boundary. There is a diffeomorphism onto its image $\varphi: D \times [-1,1] \rightarrow M$ such that $\varphi(\cdot,0) = id_D$ and $\varphi(D \times [-1,1]) \cap \Sigma = \varphi(\partial D \times [-1,1])$. Then compress $\Sigma$ along the disc $D$ as follows. Remove the band $\varphi(\partial D \times [-1,1])$ and glue two discs $D^\pm := \varphi(D \times \set{\pm 1})$ to $\Sigma$ along the corresponding curve $\varphi(\partial D \times \set{\pm 1})$; see Figure~\ref{fi:compression}. Notice that $\varphi(D \times [-1,1])$ is a 1-handle in $M$, whose boundary corresponds to $\varphi(\partial D \times [-1,1]) \cup D^\pm$. Therefore, the compression, which consists in substituting $\varphi(\partial D \times [-1,1])$ by $D^\pm$, yields a new surface $\Sigma'$ homologous to $\Sigma$ in $M$.
\begin{figure}[ht]
	\centering
    \def\svgwidth{1\columnwidth}
\begingroup%
  \makeatletter%
  \providecommand\color[2][]{%
    \errmessage{(Inkscape) Color is used for the text in Inkscape, but the package 'color.sty' is not loaded}%
    \renewcommand\color[2][]{}%
  }%
  \providecommand\transparent[1]{%
    \errmessage{(Inkscape) Transparency is used (non-zero) for the text in Inkscape, but the package 'transparent.sty' is not loaded}%
    \renewcommand\transparent[1]{}%
  }%
  \providecommand\rotatebox[2]{#2}%
  \newcommand*\fsize{\dimexpr\f@size pt\relax}%
  \newcommand*\lineheight[1]{\fontsize{\fsize}{#1\fsize}\selectfont}%
  \ifx\svgwidth\undefined%
    \setlength{\unitlength}{572.5984252bp}%
    \ifx\svgscale\undefined%
      \relax%
    \else%
      \setlength{\unitlength}{\unitlength * \real{\svgscale}}%
    \fi%
  \else%
    \setlength{\unitlength}{\svgwidth}%
  \fi%
  \global\let\svgwidth\undefined%
  \global\let\svgscale\undefined%
  \makeatother%
  \begin{picture}(1,0.2029703)%
    \lineheight{1}%
    \setlength\tabcolsep{0pt}%
    \put(0,0){\includegraphics[width=\unitlength,page=1]{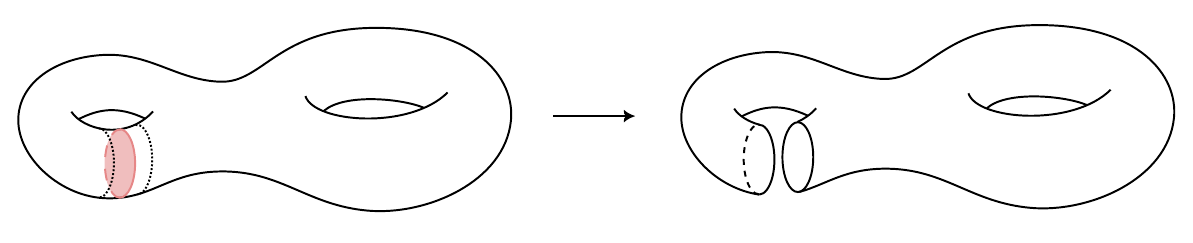}}%
    \put(0.11087653,0.01003014){\color[rgb]{0.90196078,0.52941176,0.52941176}\makebox(0,0)[lt]{\lineheight{1.25}\smash{\begin{tabular}[t]{l}$\gamma$\end{tabular}}}}%
    \put(0.05402942,0.06698752){\color[rgb]{0.90196078,0.52941176,0.52941176}\makebox(0,0)[lt]{\lineheight{1.25}\smash{\begin{tabular}[t]{l}$D$\end{tabular}}}}%
    \put(0.38023909,0.18718287){\color[rgb]{0,0,0}\makebox(0,0)[lt]{\lineheight{1.25}\smash{\begin{tabular}[t]{l}$\Sigma$\end{tabular}}}}%
    \put(0.94752623,0.18831466){\color[rgb]{0,0,0}\makebox(0,0)[lt]{\lineheight{1.25}\smash{\begin{tabular}[t]{l}$\Sigma'$\end{tabular}}}}%
  \end{picture}%
\endgroup%

	\caption{Compressing a surface.}
	\label{fi:compression}
\end{figure}

The compression of $\Sigma$ along a disc reduces its complexity. Namely, if $\gamma$ is not separating in $\Sigma$, then the compression simply reduces the genus of the compressed connected component of $\Sigma$ by one. On the other hand, if $\gamma$ is separating in $\Sigma$, then the compression splits the corresponding connected component into two new connected components. Each of them has strictly lower genus than the compressed connected component of $\Sigma$. Hence, by iterating the procedure finitely many times, we finally obtain an incompressible surface $\Sigma'$.\medskip

When the ambient manifold is simply connected at infinity on contractible curves, such a compression may be performed in a localised manner.

\begin{proposition}\label{pr:localisedlooptheorem}
	Let $M$ be an orientable 3-manifold. Suppose $M$ is simply connected at infinity on contractible curves. Let $B \subset M$ be a compact subset, and consider a compact subset $K \subset M$ containing $B$ such that
	\begin{equation} \label{eq:simply_connected_infinity}
		\ker{(\pi_1(M-K) \rightarrow \pi_1(M))} = \ker{(\pi_1(M-K) \rightarrow \pi_1(M-B))}.
	\end{equation}
	Let $\Sigma \subset M-K$ be a compressible embedded orientable surface. Then $\Sigma$ can be compressed into an incompressible surface $\Sigma' \subset M-B$ so that the surface obtained at each step of the compression is contained in $M-B$.
\end{proposition}

\begin{proof}
	Suppose that $\Sigma' \subset M - B$ is the result of compressing $\Sigma$ a number $k$ of times. One can write $\Sigma'$ as the union of $\Sigma_0 = \Sigma' \cap \Sigma$, the part corresponding to the original surface $\Sigma$, which is contained in $M-K$, and some discs $D^\pm_1, \dots, D^\pm_k \subset M-B$ glued to $\Sigma_0$ during the former compressions.
	
	\begin{figure}[ht]
		\centering
    \def\svgwidth{1\columnwidth}
    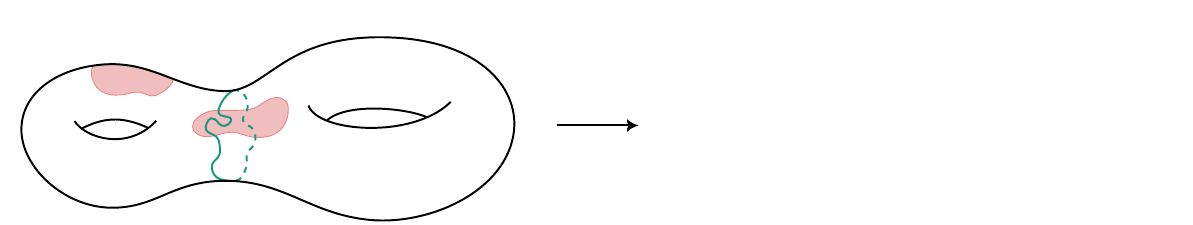

		\caption{Deforming a loop.}
		\label{fi:deforming_loop}
	\end{figure}

	Now, take a closed curve $\gamma \subset \Sigma$ representing a nontrivial element of $\ker{(\pi_1(\Sigma') \rightarrow \pi_1(M))}$. If $\gamma$ intersects some disc $D^\pm_i$, one can homotope every arc of $\gamma \cap D^\pm_i$ to the boundary $\partial D^\pm_i$, and push them slightly beyond, so that the resulting curve avoids $D^\pm_i$; see Figure \ref{fi:deforming_loop}. Repeating the procedure for each disc, the closed curve $\gamma$ can be homotoped to a closed curve $\gamma'$ lying in $\Sigma_0$. Now, since $\gamma'$ is a contractible loop contained in $\Sigma_0 \subset M - K$, the loop $\gamma'$ is already contractible in $M-B$ by the relation (\ref{eq:simply_connected_infinity}). We obtain
	\begin{equation*}
		\ker{(\pi_1(\Sigma') \rightarrow \pi_1(M))} = \ker{(\pi_1(\Sigma') \rightarrow \pi_1(M-B))}.
	\end{equation*}
	Now, we conclude by applying the Loop Theorem to the surface $\Sigma'$ contained in $M-B$, which gives an embedded disc $D \subset M-B$ such that $\Sigma' \cap D = \partial D$ is a closed curve representing a nontrivial element in $\ker{(\pi_1(\Sigma') \rightarrow \pi_1(M-B))}$. Thus, the result of the compression of $\Sigma'$ along $D$ is again contained in $M-B$.
\end{proof}

The next result provides a decomposition of the manifolds into compact domains along incompressible surfaces.

\begin{proposition} \label{pr:prime_decomposition}
	Let $M$ be a complete orientable Riemannian 3-manifold. Suppose that $M$ is simply connected at infinity on contractible curves. Then $M$ admits an exhaustion by compact connected domains
	\begin{equation*}
		K_1 \subset K_2 \subset \cdots \subset M
	\end{equation*}
	with $K_i \subset \mathring{K}_{i+1}$ and $M = \cup_i K_i$, such that $\cup_i\partial K_i$ is a locally finite disjoint collection of embedded orientable closed incompressible surfaces.
\end{proposition}

\begin{proof}
	Let $x \in M$ and consider the singleton $K_0 = \set{x}$. Suppose we have constructed $K_0 \subset \dots \subset~K_i$, with $K_j \subset U(K_j , 1) \subset K_{j+1}$, where $U(K_j , 1)$ is the closed 1-neighbourhood of $K_j$ (see Definition~\ref{de:metric_neighbourhood}~(\ref{eq:metric_neighbourhood})), such that $\partial K_j$ is a closed incompressible surface. Take a closed metric ball~$B$ containing $U(K_i , 1)$. Since $M$ is simply connected at infinity on contractible curves, there is a compact subset~$K$ containing $B$ such that
	\begin{equation*}	
		\ker{(\pi_1(M-K) \rightarrow \pi_1(M))} = \ker{(\pi_1(M-K) \rightarrow \pi_1(M-B))}.
	\end{equation*}
	Now, take a metric ball $B'$ containing $K$. After approximating the distance function $d(x,\cdot)$ by a smooth function and slightly perturbing the radius of $B'$, we can suppose that the boundary of $B'$ is an embedded orientable surface $\Sigma$. Since the surface $\Sigma$ lies in $M-K$, we can use Proposition~\ref{pr:localisedlooptheorem} to compress it inside $M-B$ into an incompressible surface $\Sigma'$ (if $\Sigma$ is already incompressible, we have~$\Sigma' = \Sigma$). Since $B \subset B'$ and $\Sigma'$ is homologous to $\Sigma$, the surface $\Sigma'$ encloses a compact connected component $K_{i+1}$ containing $B$. The boundary $\partial K_{i+1}$ of $K_{i+1}$ lies in $\Sigma'$, and thus, is incompressible.
	
	Finally, since $K_i \subset U (K_i , 1) \subset B \subset K_{i+1}$, we conclude that the compact subset $K_{i+1}$ contains the closed metric ball $B(x,i+1)$. Hence, $\cup_i K_i = M$.
\end{proof}

\section{Fill radius and number of ends} \label{se:number_ends}

Recall the following notion.

\begin{definition}
	Let $X$ be a connected locally finite simplicial complex. Given a subcomplex~$L \subset X$, denote by $n(L)$ the number of connected components of $L$. The \emph{number of topological ends} of $X$ is defined as
\begin{equation*}
	e(X) := \sup\set{n(X-K) \mid K \subset X \text{ finite subcomplex such that } X-K \text{ has no finite component}}.
\end{equation*}
\end{definition}

By \cite[Theorem 3]{Epstein_1961}, if $G$ is a group acting cocompactly by covering transformations on a simplicial complex $\bar{X}$, then the number of ends of $\bar{X}$ depends uniquely on the group $G$. This implies that the number of ends of $\bar{X}$ is a group invariant of $G$.

\begin{definition} \label{de:number_ends_group}
	The \emph{number of ends} of a finitely generated group $G$ is defined as $e(G) := e(\bar{X})$, where $\bar{X} \rightarrow X$ is any regular covering of a finite simplicial complex $X$ with covering transformation group $G$.
\end{definition}

In particular, the number of ends of a finitely generated group $G$ coincides with the number of topological ends of its Cayley graph. Similarly, the number of ends of the fundamental group $\pi_1(M)$ of a closed manifold $M$ coincides with the number of topological ends of its universal cover~$\tilde{M}$.

\begin{remark}
	An easy consequence of the definition is that $e(G) = 0$ if and only if $G$ is a finite group. Using the theory of covering spaces, it can be shown \cite[Theorem 12]{Epstein_1961} that $e(G) = 2$ if and only if $G$ is virtually infinite cyclic (that is, $G$ contains an infinite cyclic subgroup of finite index). More generally, using the theory of covering spaces, one can prove that the number of ends that a group may have is either 0, 1, 2 or infinite (see \cite[Theorem 10]{Epstein_1961} for a proof).
\end{remark}

By \cite{Ramachandran_Wolfson_2010}, the fundamental group of a closed manifold of any dimension, whose Riemannian universal cover has bounded fill radius cannot contain finitely generated subgroups with one end, and as a consequence, the fundamental group of such a manifold is virtually free (recall that a group $G$ is \emph{virtually free} if it contains a finite index free subgroup). Recently, the same strategy was used in \cite{CLL_2023} to prove a classification of closed manifolds admitting a metric of positive scalar curvature in dimensions 4 and 5, under some additional topological assumptions.\medskip

In this section, we derive the main result of \cite{Ramachandran_Wolfson_2010} under a weaker hypothesis (the fill radius of $\tilde{M}$ is not necessarily bounded, but has at most $c$-linear growth at infinity with $c < \frac{1}{3}$), adapting ideas present in \cite{Ramachandran_Wolfson_2010, CLL_2023, Gromov_Lawson_1983}. Note that the results in this section are valid for manifolds of any dimension.

\begin{theorem} \label{th:Ramachandran_Wolfson}
	Let $M$ be a complete Riemannian manifold, and consider its Riemannian universal cover $\tilde{M}$. Suppose the fill radius of $\tilde{M}$ has at most $c$-linear growth at infinity, with $c < \frac{1}{3}$.
		
	Then, the fundamental group $\pi_1(M)$ does not contain any finitely generated subgroup with exactly one end.
\end{theorem}

We start by proving the following version of \cite[Corollary 10.11]{Gromov_Lawson_1983} for manifolds whose Riemannian universal cover has fill radius with a certain growth at infinity.

\begin{lemma} \label{le:Gromov_Lawson_Lemma}
	Let $M$ be a complete Riemannian manifold, and denote by $p: \tilde{M} \rightarrow M$ its universal Riemannian cover.
	Let $x \in M$ be a point. Let $\beta \geq 0$ and $c > 0$ be constants, and suppose that either $\beta < 1$, or $\beta = 1$ and $c < \frac{1}{2}$. Suppose there is a constant $R'_0 \geq 0$ such that if $R \geq R'_0$ then, for every closed curve $\gamma$ lying in $B(x,R)$ and contractible in $M$, any of its lifts $\tilde{\gamma}$ to $\tilde{M}$ verifies
	\begin{equation*}
		\fillrad{(\tilde{\gamma})} < c R^\beta.
	\end{equation*}
	Let $Z \subset M$ be a path-connected compact subset and consider a path-connected component $\bar{Z}$ of~$p^{-1}(Z) \subset \tilde{M}$. Then, there is a constant $R''_0 \geq R'_0$ such that if $R \geq R''_0$, every path-connected component $C_R$ of the level set $\partial U(\bar{Z},R)$ satisfies
	\begin{equation*}
		\diam(C_R) < 6c\, (d(x,Z) + \diam{(Z)} + R)^\beta .
	\end{equation*}
\end{lemma}

\begin{proof}
	Let $R''_0 \geq R'_0$ be a constant to be determined later. Let $R \geq R'_0$ and denote $L := d(x,Z) + \diam{(Z)} + R$. Fix a connected component $C_R$ of the level set $\partial U(\bar{Z},R)$ and consider two points~$z_1, z_2 \in C_R$. We shall prove that $z_1$ and $z_2$ lie within a distance at most $6cL^\beta$. Let~$\tilde{\eta}$ be a curve in~$C_R$ joining $z_1$ to $z_2$. Join each point $z_i$ to $\bar{Z}$ by a minimal geodesic $\tilde{\eta}_i$, and join the endpoints of $\tilde{\eta}_1$ and $\tilde{\eta}_2$ lying in $\bar{Z}$ by a curve $\tilde{\eta}'$ lying in $\bar{Z}$; see Figure \ref{fi:diameter}. The concatenation~$\tilde{\gamma} = \tilde{\eta}_1 * \tilde{\eta} * \tilde{\eta}_2 * \tilde{\eta}'$ is a closed curve contained in $U(\bar{Z},R)$. Now, consider the projection $\gamma$ of $\tilde{\gamma}$ to~$M$. Then the closed curve $\gamma$ lies in the closed metric ball $B(x,d(x,Z) + \diam{(Z)}+R)$. Therefore,
	\begin{equation*}
		\fillrad{(\tilde{\gamma})} < c\, (d(x,Z) +\diam{(Z)} + R)^\beta = c L^\beta.
	\end{equation*}
	That is, there exists an immersed disc $D \subset \tilde{M}$ with $\partial D = \tilde{\gamma}$ such that any point $a \in D$ satisfies $d(a,\tilde{\gamma}) < c L^\beta$.
	
	\begin{figure}[ht]
		\centering
    \def\svgwidth{0.7\columnwidth}
    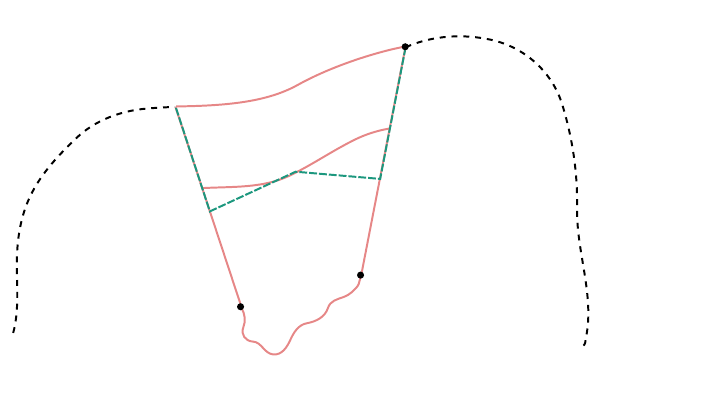

		\caption{Scheme of the proof of Lemma \ref{le:Gromov_Lawson_Lemma}.}
		\label{fi:diameter}
	\end{figure}
	
		Now, consider the curve $\tilde{\eta}'' = D \cap \partial U(\bar{Z},R-c L^\beta)$. Since either $\beta < 1$ or $\beta = 1$ and $c < \frac{1}{2}$, for~$R''_0$ large enough, if $R \geq R''_0$ then every point $a \in \tilde{\eta}''$ satisfies
		\begin{equation*}
			d(a,\tilde{\eta}') \geq d(a,\bar{Z}) = R - cL^\beta \geq cL^\beta,
		\end{equation*}
		and clearly $d(a,\tilde{\eta}) \geq cL^\beta$.
		Therefore, there is a point $a \in D$ lying within a distance less than $cL^\beta$ from both curves $\tilde{\eta}_1$ and $\tilde{\eta}_2$. That is, there are points $b_i \in \tilde{\eta}_i$ such that $d(a,b_i) < cL^\beta$, for $i = 1,2$. Since $d(b_i,z_i) < 2cL^\beta$ (otherwise the inequality $d(a,b_i) < cL^\beta$ would not hold), we obtain
		\begin{equation*}
			d(z_1 , z_2) \leq d(z_1, b_1) + d(b_1,a) + d(a,b_2) + d(b_2,z_2) < 6cL^\beta.
		\end{equation*}
		That is,
		\begin{equation*}
			d(z_1,z_2) < 6c\, (d(x,Z) + \diam{(Z)} + R)^\beta.
		\end{equation*}
\end{proof}

Now, we prove Theorem \ref{th:Ramachandran_Wolfson}.

\begin{proof}[Proof of Theorem \ref{th:Ramachandran_Wolfson}]
	Let $x \in M$ be a point and $R'_0 \geq 0$ such that if $\gamma$ is a contractible closed curve lying in the closed metric ball $B(x,R)$ for $R \geq R'_0$, then any of its lifts $\tilde{\gamma}$ to the universal cover $\tilde{M}$ satisfies
	\begin{equation*}
		\fillrad{(\tilde{\gamma})} < cR,
	\end{equation*}
	with $c < \frac{1}{3}$. Suppose that $G$ is a finitely generated subgroup of $\pi_1(M,x)$ with exactly one end. In particular, the subgroup $G$ is infinite. Consider a collection of closed curves $\eta_1, \dots, \eta_k$ based at $x$ representing the generators of $G$. Now, consider the lift $\bar{X}$ of $X = \cup_i\eta_i$ to the universal cover~$\tilde{M}$, and fix a lift $\tilde{x} \in \bar{X}$ of the point $x$. Notice that $\bar{X}$ is homeomorphic to the Cayley graph of $G$ associated to the generating set represented by the homotopy classes of the curves $\set{\eta_i}$.\\
	
	We shall need the following result.
	
	\begin{lemma} \label{le:bounded_distance}
		There exists a constant $H > 0$ such that for any minimising geodesic $\tilde{\gamma}$ joining two points $y_1, y_2$ lying in $\bar{X}$, we have
		\begin{equation*}
			\max_{y \in \tilde{\gamma}} d(y,\bar{X}) \leq H.
		\end{equation*}
	\end{lemma}
	
	\begin{proof}
		We argue by contradiction. Suppose that for any $H > 0$, there is a minimal geodesic $\tilde{\gamma}$ joining two points $y_1$ and $y_2$ lying in $\bar{X}$ and a point $y_0 \in \tilde{\gamma}$ such that $d(y_0,\bar{X}) = H$; see Figure \ref{fi:height}. In particular, $d(y_0, y_j) \geq H$ for $j = 1,2$, and $d(y_1,y_2) \geq 2H$.
		
		\begin{figure}[ht]
			\centering
    \def\svgwidth{0.7\columnwidth}
    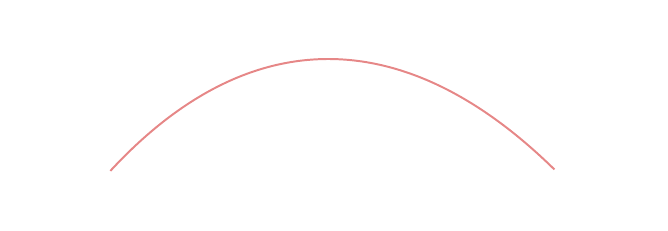

			\caption{Minimising curve $\tilde{\gamma}$ joining two points $y_1$ and $y_2$ lying in $\bar{X}$.}
			\label{fi:height}
		\end{figure}
		
		Let $y'_j$ be the intersection point of $\tilde{\gamma}$ with the level set $\partial U(\bar{X},\frac{H}{2})$ which is closest to $y_j$, for$j = 1,2$. The points $y'_1$ and $y'_2$ must lie in the same connected component of $\partial U(\bar{X},\frac{H}{2})$, otherwise the concatenation of $\tilde{\gamma}$ with a curve lying in $\bar{X}$ joining $y_1$ with $y_2$, which is contractible, would have nontrivial intersection with the cycle $\partial U(\bar{X},\frac{H}{2})$. Applying Lemma \ref{le:Gromov_Lawson_Lemma} to $\bar{Z} = \bar{X}$, $\beta = 1$ and $c < \frac{1}{3}$, we have that for $H$ large enough,
		\begin{equation*}
			d(y_1,y_2) \leq d(y_1,y'_1) + d(y'_1,y'_2) + d(y'_2,y_2) < H + 6c\,(\diam{(X)} + \tfrac{H}{2}).
		\end{equation*}
		Since $c < \frac{1}{3}$, this contradicts the fact that $d(y_1, y_2) \geq 2H$, for $H$ sufficiently large.
	\end{proof}
	
	Now, let $R \geq R'_0$. Fix a fundamental domain $\Delta$ of $\bar{X}$ for the action of $G$. Since $e(\bar{X}) = e(G) = 1$, there is a compact $K \subset \bar{X}$ containing $U(\Delta, R+H) \cap \bar{X}$ such that $\bar{X} - K$ has a unique (unbounded) connected component.
	
	Take a minimising curve $\tilde{\gamma}$ of $\tilde{M}$ joining two points $y_1$ and $y_2$ in $\bar{X}$ at distance at least~$2T$ in~$\tilde{M}$, where $T > \diam{(K)} + H$. Choose a point $y_0 \in \tilde{\gamma}$ with $d(y_0,y_j) \geq T$ for $j=1,2$ and a point $\bar{y}_0 \in \bar{X}$ at minimal distance from $y_0$. By Lemma \ref{le:bounded_distance}, the minimising curve $\tilde{\gamma}$ is at distance at most~$H$ from~$\bar{X}$. Thus, $d(y_0, \bar{y}_0) \leq H$. Translating the curve $\tilde{\gamma}$ by an element of $G$ if necessary, we can assume that $\bar{y}_0$ lies in $\Delta$. It follows that 
	\begin{equation*}
		B(y_0,R) \cap \bar{X} \subset B(\bar{y}_0,R+H) \cap \bar{X} \subset K.
	\end{equation*}
	Now, by the triangle inequality,
	\begin{equation*}
		d(\bar{y}_0,y_j) \geq d(y_0,y_j) -d(y_0,\bar{y}_0) \geq T-H > \diam{(K)}.
	\end{equation*}
	Thus, both $y_1$ and $y_2$ lie in $\bar{X}-K$ and can be joined by a curve $\tau$ lying in $\bar{X} - K$ by construction. Since $B(y_0,R) \cap \bar{X} \subset K$, the curve $\tau$ lies outside $B(y_0,R)$.
	
		\begin{figure}[htb]
		\centering
    \def\svgwidth{0.7\columnwidth}
    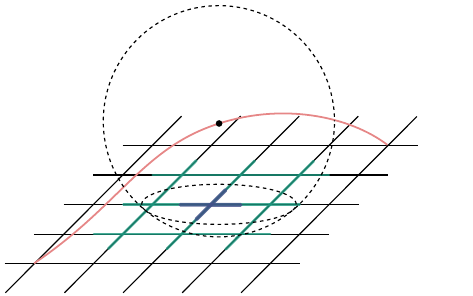

		\caption{Sketch of the proof of Theorem \ref{th:Ramachandran_Wolfson}.}
		\label{fi:ball}
	\end{figure}
	
	Consider the intersection point $z_j$ of $\tilde{\gamma}$ with the sphere $\partial B(y_0,R)$ which is the closest to the point $y_j$ for $j=1,2$. The points $z_1$ and $z_2$ must lie in the same connected component of $\partial B(y_0,R)$, otherwise the concatenation $\tilde{\gamma} * \tau$ would have a nontrivial intersection number with the cycle $\partial B(y_0,R)$, which is absurd since $\tilde{\gamma} * \tau$ is contractible. By Lemma \ref{le:Gromov_Lawson_Lemma} applied to $\bar{Z} = \set{y_0}$, $\beta = 1$ and $c < \frac{1}{3}$, we have
	\begin{equation} \label{eq:distance_minimising_curve}
		d(z_1,z_2) < 6c \, (d(x,p(y_0)) + R)
	\end{equation}
	where $p: \tilde{M} \rightarrow M$ is the universal covering. Now, recall that $p(y_0)$ is at distance at most $H$ from $X$ by Lemma \ref{le:bounded_distance}. Thus, the distance between $p(y_0)$ and the basepoint $x$ of $X$ is bounded uniformly in $R$. More precisely,
	\begin{equation*}
		d(x,p(y_0)) \leq \diam{(X)} + H.
	\end{equation*}
	Since the curve $\tilde{\gamma}$ is minimising, we have $d(z_1,z_2) = 2R$. By taking $R$ large enough in the inequality~(\ref{eq:distance_minimising_curve}), we obtain a contradiction with $c < \frac{1}{3}$.
\end{proof}

Finally, we recover the main result of \cite{Ramachandran_Wolfson_2010} under more general assumptions (that is, not just for closed manifolds whose universal cover has bounded fill radius).
\begin{corollary} \label{co:Ramachandran_Wolfson}
	Let $M$ be a complete Riemannian manifold with finitely presented fundamental group. Denote by $\tilde{M}$ its Riemannian universal cover. Suppose the fill radius of $\tilde{M}$ has at most $c$-linear growth at infinity, with $c < \frac{1}{3}$. Then $\pi_1(M)$ is virtually free.
\end{corollary}

\begin{proof}
	We follow the argument of \cite{Ramachandran_Wolfson_2010}. By Theorem \ref{th:Ramachandran_Wolfson}, $\pi_1(M)$ does not contain finitely generated subgroups of exactly one end. Since $\pi_1(M)$ is finitely presented, then it is accessible by \cite{Dunwoody_1985}. Recall that a group is accessible if it is the fundamental group of a graph of groups such that every edge group is finite and every vertex group is at most one-ended. Finally, an accessible group without one-ended subgroups is virtually free, by a result of Serre \cite[Chapter 2, Section 2.6, Proposition 11]{Serre_1980}.
\end{proof}

\begin{remark}
	In particular, if the fill radius of the universal cover of a complete Riemannian manifold with finitely presented fundamental group is uniformly bounded above, then its fundamental group is virtually free.
\end{remark}

\section{Fill radius and aspherical summands} \label{se:aspherical_summands}

Another consequence of Theorem \ref{th:Ramachandran_Wolfson} is that manifolds which decompose as the connected sum of a manifold with an aspherical summand do not admit complete metrics such that the Riemannian universal cover has at most $c$-linear growth at infinity with $c < \frac{1}{3}$. It actually follows from the following well-known fact; see for instance \cite[Proposition 16.4.1]{Geoghegan_2008}.

\begin{proposition} \label{pr:aspherical_oneend}
	If $P$ is a closed aspherical $n$-manifold with $n \geq 2$, then $\pi_1(P)$ has exactly one end.
\end{proposition}

\begin{proof}
	Endow $P$ with a finite simplicial complex structure and consider its universal cover $\tilde{P}$. By Definition \ref{de:number_ends_group}, we have $e(\pi_1(P)) = e(\tilde{P})$. Let $K \subset \tilde{P}$ be a finite subcomplex whose complementary has no finite component. The long exact sequence corresponding to the pair $(\tilde{P},\tilde{P}-K)$ for reduced cohomology (with coefficients in a field) may be written
	\begin{equation*}
	\dots \leftarrow \underset{\underset{0}{\shortparallel}}{\tilde{H}^1(\tilde{P})} \leftarrow  H^1(\tilde{P},\tilde{P}-K) \leftarrow \tilde{H}^0(\tilde{P}-K) \leftarrow \underset{\underset{0}{\shortparallel}}{\tilde{H}^0(\tilde{P})} \leftarrow  H^0(\tilde{P},\tilde{P}-K) \leftarrow 0
\end{equation*}
and the contractibility of $\tilde{P}$ gives an isomorphism $H^1
(\tilde{P},\tilde{P}-K) \simeq \tilde{H}^0(\tilde{P}-K)$. Now, Poincar\'e's duality gives an isomorphism $H_{n-1}(\tilde{P}) \simeq H_c^1(\tilde{P}) = \varprojlim_K H^1(\tilde{P},\tilde{P}-K)$ (see \cite[Section 3.3]{Hatcher_2002} for the second isomorphism). Therefore,
\begin{equation*}
	H_{n-1}(\tilde{P}) = \varprojlim_K \tilde{H}^0(\tilde{P},\tilde{P}-K).
\end{equation*}
 Since $\tilde{P}$ is contractible, the former space must vanish. Now, since the number of ends of $\tilde{P}$ coincides with $\dim\left( \varprojlim_K H^0(\tilde{P}-K) \right) = \dim\left( \varprojlim_K \tilde{H}^0(\tilde{P}-K) \right)+1$ (see for instance \cite[Theorem 1]{Epstein_1961}), we finally conclude that $e(\tilde{P}) = 1$.
\end{proof}

We deduce the following result from Theorem \ref{th:Ramachandran_Wolfson} and Proposition \ref{pr:aspherical_oneend}.

\begin{corollary} \label{co:no_aspherical_summands}
	Let $P$ be a closed aspherical $n$-manifold and $N$ be an arbitrary $n$-manifold with $n\geq 2$. Then the connected sum $M = P \# N$ does not admit any complete Riemannian metric such that the fill radius of $\tilde{M}$ has at most $c$-linear growth at infinity, with $c < \frac{1}{3}$.
\end{corollary}

By \cite{Gromov_Lawson_1983}, closed aspherical 3-manifolds do not support Riemannian metrics of positive scalar curvature. The same statement for any dimension was conjectured by Gromov \cite{Gromov_1986}.

\begin{conjecture}\label{conj:Gromov_aspherical}
	A closed aspherical $n$-manifold does not admit any Riemannian metric with positive scalar curvature.
\end{conjecture}

Conjecture \ref{conj:Gromov_aspherical} was solved for $n \in \set{4,5}$ by Chodosh--Li \cite{Chodosh_Li_2024}, and independently by Gromov \cite{Gromov_2020}. Gromov \cite[Section 4, Conjecture 12]{Gromov_2017}, \cite[Section 3.2]{Gromov_2023} also conjectured a stronger version of Conjecture \ref{conj:Gromov_aspherical}, involving the notion of $\Q$-essential manifold.

\begin{definition}
	A closed $n$-manifold $M$ is $\Q$-\textit{essential} if the classifying map $f: M \rightarrow K(\pi_1(M),1)$ induces a nontrivial homomorphism $f_* : H_n(M;\Q) \rightarrow H_n(K;\Q)$ in top dimensional rational homology, that is,
	\begin{equation*}
		f_* [M] \neq 0 \in H_n(K; \Q).
	\end{equation*}
\end{definition}
Clearly, any closed aspherical manifold is $\Q$-essential.

\begin{conjecture}[\cite{Gromov_2017,Gromov_2023}] \label{conj:Gromov_Q_essential}
	A closed $\Q$-essential $n$-manifold does not admit any Riemannian metric with positive scalar curvature.
\end{conjecture}

The results in Section \ref{se:number_ends} allow us to prove a weaker version of Conjecture \ref{conj:Gromov_Q_essential} involving the fill radius.

\begin{proposition} \label{pr:Qessential}
	Let $M$ be a closed $\Q$-essential $n$-manifold with $n \geq 2$. Then $M$ does not admit any Riemannian metric such that the fill radius of its universal cover $\tilde{M}$ is bounded.
\end{proposition}

\begin{proof}
	Suppose that $\tilde{M}$ has bounded fill radius. Then Corollary \ref{co:Ramachandran_Wolfson} implies that the fundamental group $G$ of $M$ is virtually free. This means that there is a finite covering $p: N \rightarrow M$ of nonzero degree $k$ such that the fundamental group $F = \pi_1(N)$ is free. Consider the classifying map $f: M \rightarrow K(G,1)$. We can lift the map $f$ to obtain the following commutative diagram
	\begin{equation*}
		\begin{tikzcd}
			N \arrow{r}{\bar{f}} \arrow[swap]{d}{p} & K(F,1) \arrow{d}{} \\
			M \arrow{r}{f} & K(G,1).
		\end{tikzcd}
	\end{equation*}
	The corresponding commutative diagram induced in $n$-dimensional rational homology is
	\begin{equation*}
		\begin{tikzcd}
			H_n(N;\Q) \arrow{r}{\bar{f}_*} \arrow[swap]{d}{p_*} & H_n(F;\Q) \arrow{d}{} \\
			H_n(M;\Q) \arrow{r}{f_*} & H_n(G;\Q).
		\end{tikzcd}
	\end{equation*}
	Since $F$ is a free group, its classifying space $K(F,1)$ is homotopy equivalent to a graph, and therefore $H_n(F;\Q) = 0$. In particular $\bar{f}_* [N] = 0$. But, on the other hand, we have
	\begin{equation*}
		(f \circ p)_* [N] = k f_* [M],
	\end{equation*}
	which implies $f_* [M] = 0$. Therefore, $M$ is not $\Q$-essential.
\end{proof}

It follows from Proposition \ref{pr:Qessential} and Theorem \ref{th:Gromov_Lawson} that $\Q$-essential closed 3-manifolds do not admit Riemannian metrics with positive scalar curvature, result that was already proved in \cite{Gromov_Lawson_1983}.

\begin{remark}
	It is not true that a closed manifold whose Riemannian universal cover has bounded fill radius is not $\Z$-essential. Indeed, $\RP^n$ is $\Z$-essential, and the fill radius of its universal cover is clearly bounded (for any metric). However, $\RP^n$ is not $\Q$-essential since $H_n(\Z/2\Z;\Q) = 0$. More generally, the latter holds for all lens spaces.
\end{remark}

\section{Proof of the decomposition theorem} \label{se:prime_decomposition}

Let us present the notion of infinite connected sum modelled on a locally finite graph, following \cite{Scott_1977}. 

\begin{definition}
	Let $\mathcal{F}$ be a family of connected 3-manifolds. A 3-manifold $M$ decomposes as a \emph{connected sum} of members of $\mathcal{F}$ modelled on a locally finite graph $\mathcal{G}$ if there is a map assigning to each vertex~$v$ a copy $M_v$ of a manifold in $\mathcal{F}$ such that $M$ is diffeomorphic to the manifold constructed as follows:

	\begin{enumerate}
		\item For each vertex $v$, construct a new manifold $Y_v$ by removing from $M_v$ a number of $\deg{(v)}$ 3-balls from its interior,
		\item For each edge $e$ joining two vertices $v$ and $v'$, glue two spherical boundary components from $\partial Y_v$ and $\partial Y_{v'}$ along an orientation-reversing diffeomorphism.
	\end{enumerate}
	See Figure \ref{fi:infinite_connected_sum}.

\begin{figure}[ht]
	\centering
    \def\svgwidth{0.9\columnwidth}
\begingroup%
  \makeatletter%
  \providecommand\color[2][]{%
    \errmessage{(Inkscape) Color is used for the text in Inkscape, but the package 'color.sty' is not loaded}%
    \renewcommand\color[2][]{}%
  }%
  \providecommand\transparent[1]{%
    \errmessage{(Inkscape) Transparency is used (non-zero) for the text in Inkscape, but the package 'transparent.sty' is not loaded}%
    \renewcommand\transparent[1]{}%
  }%
  \providecommand\rotatebox[2]{#2}%
  \newcommand*\fsize{\dimexpr\f@size pt\relax}%
  \newcommand*\lineheight[1]{\fontsize{\fsize}{#1\fsize}\selectfont}%
  \ifx\svgwidth\undefined%
    \setlength{\unitlength}{691.65354331bp}%
    \ifx\svgscale\undefined%
      \relax%
    \else%
      \setlength{\unitlength}{\unitlength * \real{\svgscale}}%
    \fi%
  \else%
    \setlength{\unitlength}{\svgwidth}%
  \fi%
  \global\let\svgwidth\undefined%
  \global\let\svgscale\undefined%
  \makeatother%
  \begin{picture}(1,0.61885246)%
    \lineheight{1}%
    \setlength\tabcolsep{0pt}%
    \put(0,0){\includegraphics[width=\unitlength,page=1]{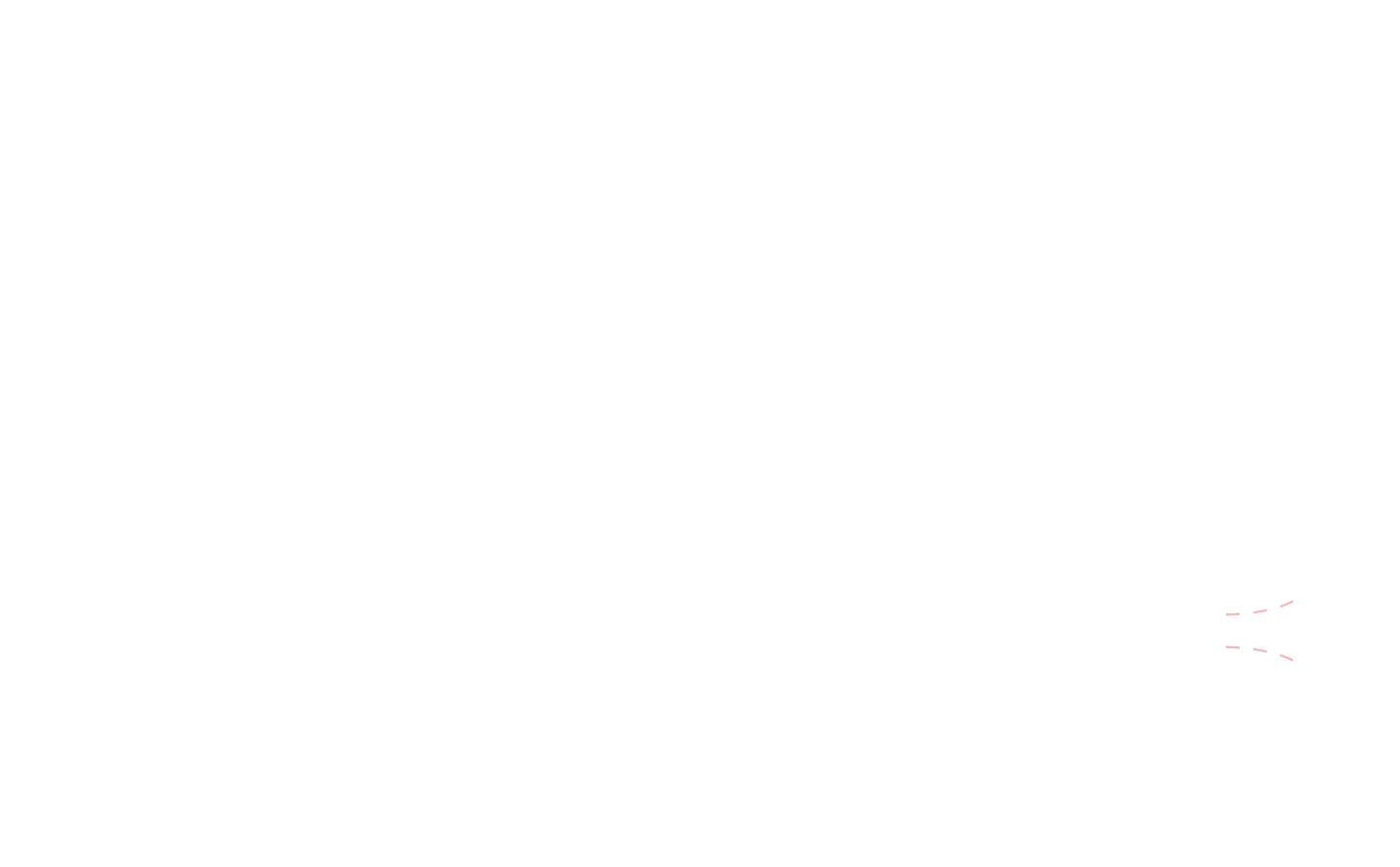}}%
    \put(0.39323947,0.37165321){\color[rgb]{0,0,0}\makebox(0,0)[lt]{\lineheight{1.25}\smash{\begin{tabular}[t]{l}$v$\end{tabular}}}}%
    \put(0.54671853,0.37264435){\color[rgb]{0,0,0}\makebox(0,0)[lt]{\lineheight{1.25}\smash{\begin{tabular}[t]{l}$v'$\end{tabular}}}}%
    \put(0.30264354,0.42298472){\color[rgb]{0,0,0}\makebox(0,0)[lt]{\lineheight{1.25}\smash{\begin{tabular}[t]{l}$M_v$\end{tabular}}}}%
    \put(0.61753864,0.42305849){\color[rgb]{0,0,0}\makebox(0,0)[lt]{\lineheight{1.25}\smash{\begin{tabular}[t]{l}$M_{v'}$\end{tabular}}}}%
    \put(0,0){\includegraphics[width=\unitlength,page=2]{infiniteconnedctedsum.pdf}}%
    \put(0.93432711,0.18305188){\color[rgb]{0,0,0}\makebox(0,0)[lt]{\lineheight{1.25}\smash{\begin{tabular}[t]{l}$\mathcal{G}$\end{tabular}}}}%
  \end{picture}%
\endgroup%

	\caption{Connected sum modelled on a graph.}
	\label{fi:infinite_connected_sum}
\end{figure}

Equivalently, a 3-manifold $M$ decomposes as a \emph{connected sum} (modelled on a locally finite graph) of members of $\mathcal{F}$ if there exists a locally finite collection of pairwise disjoint 2-spheres embedded in~$M$ such that cutting $M$ along these 2-spheres and then capping off each new spherical boundary component by a 3-ball results in a disjoint collection of manifolds belonging to $\mathcal{F}$.
Clearly, the resulting manifold depends on the graph $\mathcal{G}$ on which it is modelled. On the other hand, if a 3-manifold decomposes as an infinite connected sum, the graph $\mathcal{G}$ on which $M$ is modelled may not be unique.

If we assume $\mathcal{G}$ to be a finite tree, then we recover the usual notion of connected sum. In any case, one can turn any infinite connected sum modelled on a graph into one modelled on a tree by adding some additional $\Sp^2 \times \Sp^1$ summands. In our decomposition result, we do not consider uniqueness, so these matters will not be of importance for us.
\end{definition}

Recall the following classification theorem for orientable closed prime 3-manifolds; see \cite{Hatcher_2004}. We reproduce also its proof since it is not too long.

\begin{theorem} \label{th:classification_prime_3manifolds}
	Let $P$ be an orientable closed prime 3-manifold. Then $P$ is either spherical, aspherical or homeomorphic to $\Sp^2 \times \Sp^1$.
\end{theorem}

\begin{proof}
	Orientable closed prime 3-manifolds $P$ are classified by their fundamental group as follows.
	
	If $\pi_1(P)$ is finite, then Perelman's resolution of 	the elliptisation conjecture \cite{Perelman_2002,Perelman_2003_a,Perelman_2003_b} implies that $P$ is a spherical manifold.
	
	Now, suppose that $\pi_1(P)$ is infinite. Then the universal cover $\tilde{P}$ of $P$ is non-compact. By the long exact homotopy sequence of a fibration \cite[Theorem 4.41]{Hatcher_2002}, we have $\pi_2(\tilde{P}) \simeq \pi_2(P)$.
	
	Suppose first that $\pi_2(P) = 0$. Then $\pi_2(\tilde{P}) = 0$ and $\tilde{P}$ is 2-homotopically connected. By Hurewicz's theorem \cite[Theorem 4.32]{Hatcher_2002}, there is an isomorphism $\pi_3(\tilde{P}) \simeq H_3(\tilde{P})$. Since $\tilde{P}$ is non-compact, we have by Poincar\'e duality that $H_3(\tilde{P};\Z)$ vanishes, so the group $\pi_3(\tilde{P})$, vanishes as well. Since the homology groups $H_k (\tilde{P};\Z)$ vanish for $k \geq 4$, we can apply Hurewicz's theorem inductively to conclude that $\pi_k(\tilde{P}) = 0$ for $k \geq 0$. Therefore, the manifold $P$ is aspherical.
	
	Finally, suppose that $\pi_2(P)$ is nontrivial. Then, by an argument based on Papakyriakopoulos' Sphere Theorem, one can prove that $P$ is homeomorphic to $\Sp^2 \times \Sp^1$; see \cite{Hempel_1976}.
\end{proof}

We finally prove Theorem \ref{th:main_general}.

\begin{proof}[Proof of Theorem \ref{th:main_general}]
By Proposition~\ref{pr:prime_decomposition}, consider an exhaustion of $M$ by compact domains $K_1 \subset K_2 \subset \dots \subset M$, whose boundaries form a locally finite collection of orientable closed connected incompressible surfaces, denoted by~$\set{\Sigma_\alpha}$. Since each surface is incompressible, $\pi_1(\Sigma_\alpha)$ is a finitely generated subgroup of $\pi_1(M)$, which cannot have exactly one end by Theorem \ref{th:Ramachandran_Wolfson}. But since $\pi_1(\Sigma_\alpha)$ is a surface group associated to a closed orientable surface, the only possibility is that $\Sigma_\alpha$ is a 2-sphere.

The result of cutting $M$ along the collection of spheres $\set{\Sigma_\alpha}$ consists of the connected components~$Y_{ij}$ of the pieces $Y_i = \overline{K_i-K_{i-1}}$. Denote by $\hat{Y}_{ij}$ the result of capping the spherical boundary components of $Y_{ij}$ off by 3-balls. 

Recall that a 3-manifold $P$ is prime if it cannot be decomposed as a non-trivial connected sum. That is, if $P \simeq P_1 \# P_2$, then either $P_1$ or $P_2$ is homeomorphic to the 3-sphere. By the Kneser--Milnor Decomposition Theorem (see \cite[Chapter 3]{Hempel_1976}, for instance), each $\hat{Y}_{ij}$ decomposes as a connected sum of prime closed 3-manifolds. Namely, there is a finite collection of disjoint spheres~$\set{\Sigma_{\beta_{ij}}}$, which can be taken inside $\intt{(Y_{ij})}$, such that $M - (\cup_\alpha \Sigma_\alpha \cup_{\beta_{ij}} \Sigma_{\beta_{ij}})$ consists of the disjoint union of prime manifolds with some punctures; see Figure \ref{fi:prime_decomposition}. Denote each of these punctured prime manifolds together with their boundary spheres by $P_{ijk}$, and the result of capping their spherical boundary components off by $\hat{P}_{ijk}$.

\begin{figure}[h]
	\centering
    \def\svgwidth{1\columnwidth}
    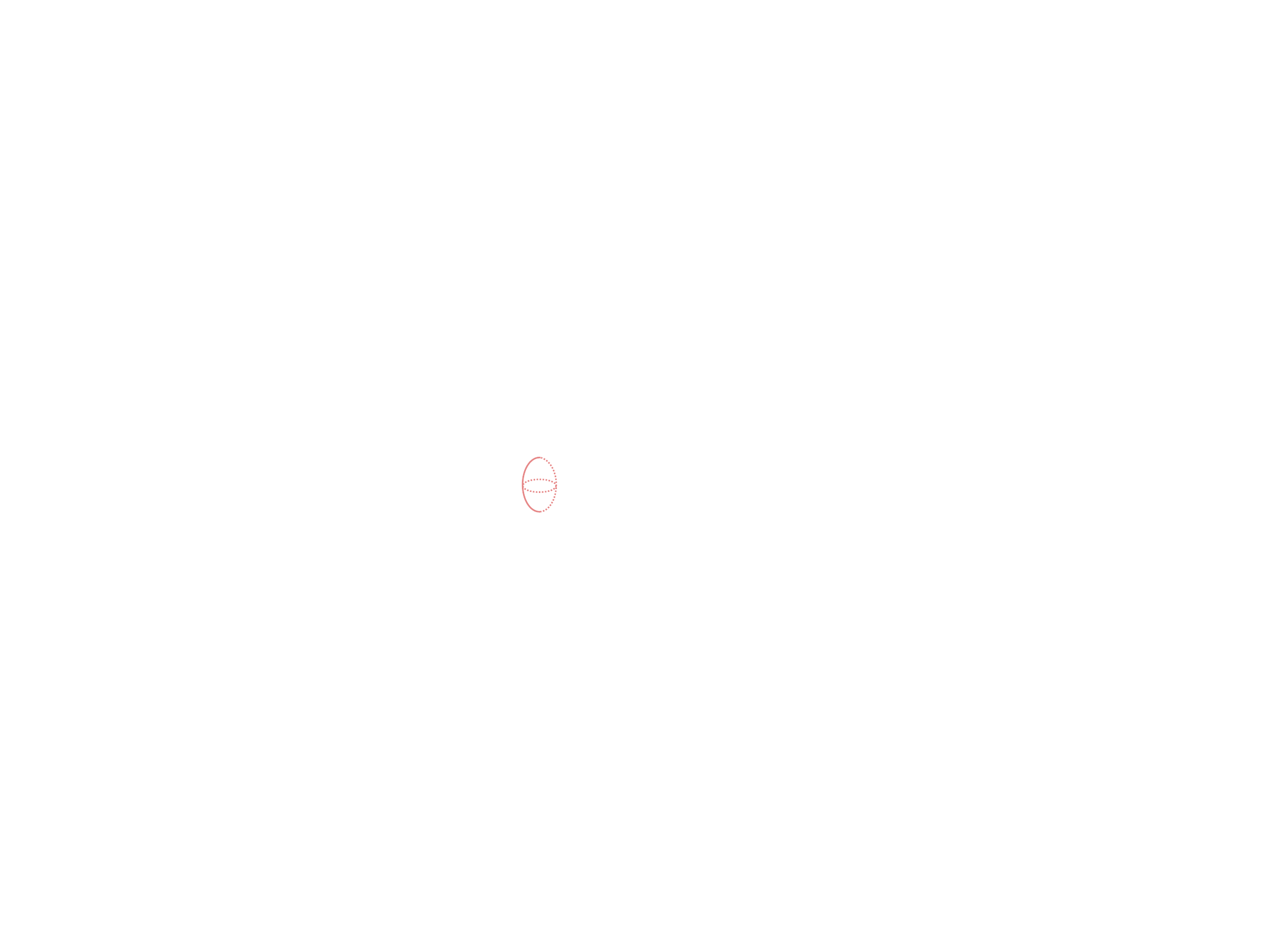

	\caption{Prime decomposition of $M$.}
	\label{fi:prime_decomposition}
\end{figure}

Now, construct a locally finite graph $\mathcal{G}$ as follows. The graph $\mathcal{G}$ has a vertex $v$ for each prime manifold $P_{ijk}$, and there is an edge $e$ joining a pair of vertices $v, v'$ for each sphere in the common boundary $\partial P_v \cap \partial P_{v'}$, where $P_v$ and $P_{v'}$ are the prime manifolds corresponding to $v$ and $v'$. By construction, the manifold~$M$ decomposes as a (possibly infinite) connected sum of the closed prime 3-manifolds $\hat{P}_{ijk}$ modelled on $\mathcal{G}$.

By Theorem \ref{th:classification_prime_3manifolds}, each closed prime 3-manifold $\hat{P}_{ijk}$ is either spherical, aspherical or homeomorphic to $\Sp^2 \times \Sp^1$. By Corollary \ref{co:no_aspherical_summands}, the manifold $\hat{P}_{ijk}$ cannot be aspherical. Therefore, $M$ is homeomorphic to a (possibly infinite) connected sum of spherical manifolds and $\Sp^2 \times \Sp^1$ modelled on $\mathcal{G}$.
\end{proof}

\section{An indecomposable 3-manifold with quadratic decay} \label{se:example_R2xS1}

Finally, as observed in \cite[Section 3.10.2]{Gromov_2023}, the manifold $\R^2 \times \Sp^1$ admits a complete Riemannian metric of positive scalar curvature with a $C$-quadratic decay for a constant $C < 64 \pi^2$. Since it does not decompose as a connected sum of spherical manifolds and $\Sp^2 \times \Sp^1$, this shows the optimality of the decay rate in Theorem \ref{th:main}.

\begin{proposition}[{\cite[Section 3.10.2]{Gromov_2023}}]
	The manifold $\R^2 \times \Sp^1$ admits a complete Riemannian metric of positive scalar curvature with $\frac{1}{2}$-quadratic decay at infinity.
\end{proposition}

\begin{proof}
Take polar coordinates $(r,\theta)$ on the first factor $\R^2$, and consider the rotationally invariant product metric
\begin{equation*}
	g = dr^2 + f(r)^2 d\theta^2 + dt^2.
\end{equation*}
By definition of the scalar curvature, the scalar curvature of $g$ is twice the Gauss curvature of $(\R^2, dr^2 + f(r)^2 d\theta^2)$. Since the metric is rotationally invariant, its scalar curvature is given by (see \cite[3.50, p.147]{GHL_2004}, for instance)
\begin{equation*}
	\scal = -2\,\frac{f''}{f}.
\end{equation*}

Now, consider the function
\begin{equation*}
	f(r) = \begin{cases}
					\sin{(r)} & \text{if } r \in [0,\frac{\pi}{4}] \\
					\varphi(r) & \text{if } r \in (\frac{\pi}{4}, 2) \\
					\sqrt{r} & \text{if } r \in [2, +\infty) 
				\end{cases}
\end{equation*}
where $\varphi$ is a concave smooth interpolation between $\sin{(r)}$ and $\sqrt{r}$. For this particular function, the scalar curvature is
\begin{equation*}
	\scal =\begin{cases}
					2 & \text{if }  r \in [0,\frac{\pi}{4}] \\
					-2\,\frac{\varphi''}{\varphi} & \text{if } r \in (\frac{\pi}{4}, 2) \\
					\frac{1}{2} \frac{1}{r^2} & \text{if } r \in [2, +\infty) 
				\end{cases}
\end{equation*}
In particular, the scalar curvature is positive with exactly quadratic decay at infinity.
\end{proof}

\begin{proposition}
	The manifold $\R^2 \times \Sp^1$ does not decompose as a connected sum of spherical manifolds and $\Sp^2 \times \Sp^1$.
\end{proposition}

\begin{proof}
	Suppose that $\R^2 \times \Sp^1$ is homeomorphic to a connected sum $M$ modelled on a locally finite graph $\mathcal{G}$ of closed prime manifolds $P_i$, each of them homeomorphic to some spherical manifold or to $\Sp^2 \times \Sp^1$. Then, Van Kampen's theorem implies
	\begin{equation*}
		\pi_1(M) \simeq *_i \, \pi_1(P_i) * \pi_1(\mathcal{G}).
	\end{equation*}
	We can always add $\Sp^2 \times \Sp^1$ summands to the connected sum to turn $\mathcal{G}$ into a locally finite tree. Hence, $\pi_1(M) \simeq *_i\pi_1(P_i)$. However, $\pi_1(\R^2 \times \Sp^1) \simeq \Z$ is torsion free, so, after permuting the summands in the connected sum, one can assume that $P_1 \simeq \Sp^2 \times \Sp^1$ and $P_i \simeq \Sp^3$ for $i \geq 2$. Notice that, since $\R^2 \times \Sp^1$ has one end, the tree $\mathcal{G}$ must have one end as well. So, after removing some $\Sp^3$ terms, the tree $\mathcal{G}$ can be assumed to be homeomorphic to a half-line. Since an infinite connected sum of 3-spheres modelled on a half-line graph is homeomorphic to $\R^3$, the manifold $M$ is homeomorphic to $(\Sp^2 \times \Sp^1) \# \R^3$, which is homotopically equivalent to the wedge sum $\Sp^2 \vee \Sp^1$. We obtain a contradiction with $\pi_2(\R^2 \times \Sp^1) = 0$.
\end{proof}

\section{The surgery theorem for infinite connected sums} \label{se:surgery}

In this section, we prove that a complete Riemannian manifold of positive scalar curvature with at most $C$-quadratic decay at infinity for some $C > 64 \pi^2$ admits a complete Riemannian metric with uniformly positive scalar curvature; see Corollary \ref{co:surgery_uniform_scal}.

\medskip

By Theorem \ref{th:main}, Corollary \ref{co:surgery_uniform_scal} follows from the following result.

\begin{theorem} \label{th:surgery}
	Let $M$ be an orientable complete Riemannian 3-manifold which decomposes as a possibly infinite connected sum of spherical manifolds and $\Sp^2 \times \Sp^1$. Then $M$ admits a complete Riemannian metric of uniformly positive scalar curvature.
\end{theorem}

The proof of Theorem \ref{th:surgery} is based on the following local construction, which provides a control on the lower bound of the scalar curvature. This construction is key to the proof of Gromov--Lawson's Surgery Theorem \cite[Theorem A]{Gromov_Lawson_1980b}, which states that the connected sum of two closed 3-manifolds of positive scalar curvature admits a metric of positive scalar curvature.

\begin{proposition}[{\cite{Gromov_Lawson_1980b}}] \label{pr:surgery_metric}
	Let $(M,g)$ be a complete Riemannian 3-manifold with uniformly positive scalar curvature $\scal \geq s_0 > 0$. Fix $x \in M$. Let $R \leq \min\set{\frac{1}{2}\inj_M(x),1}$, where $\inj_M(x)$ denotes the injectivity radius of $M$ at $x$. Let $\alpha \in (0,1)$ be a constant. Then there exist a radius $r \in (0,R)$ and a Riemannian metric $g'$ on the punctured ball $B(x,R)-\set{x}$ satisfying the following properties:
	\begin{enumerate}
		\item the new Riemannian metric $g'$ coincides with $g$ on $\partial B(x,R)$;
		\item the punctured ball $(B(x,r) - \set{x},g')$ is isometric to the standard Riemannian cylinder $\Sp^2(r) \times \R$;
		\item the scalar curvature of $g'$ satisfies $\scal_{g'} \geq \alpha s_0 > 0$.
	\end{enumerate}
\end{proposition}

Let us prove Theorem \ref{th:surgery}.

\begin{proof}[Proof of Theorem \ref{th:surgery}]
Let $\mathcal{G}$ be a locally finite graph such that $M$ decomposes as a connected sum of spherical manifolds and $\Sp^2 \times \Sp^1$ modelled on $\mathcal{G}$. For each vertex $v$ of $\mathcal{G}$, endow the corresponding manifold $M_v$ with a metric of $\scal \geq s_0 > 0$. For every edge $e$ joining two vertices $v^+$ and $v^-$ of $\mathcal{G}$, let $(x_e^+,x_e^-)$ be a pair of points such that $x_e^\pm \in M_{v^\pm}$. For every edge $e$ of $\mathcal{G}$, let $R_e > 0$ be a radius satisfying the following two conditions:
	\begin{enumerate}
		\item $R_e \leq \min\set{\frac{1}{2}\inj_{M_{v^+}}(x^+_e), \frac{1}{2}\inj_{M_{v^-}}(x^-_e), 1}$;
		\item the balls $B(x^\pm_e,R_e)$ are pairwise disjoint.
	\end{enumerate}
	Now, fix $\alpha \in (0,1)$. For each edge $e$ of $\mathcal{G}$, let $r_e \in (0,R_e)$ and $(g^\pm_e)'$ be the radius and the Riemannian metric on $B(x_e^\pm,R_e)$ given by Proposition \ref{pr:surgery_metric}. For each edge $e$, glue $(B(x_e^+,R_e)-B(x_e^+,r_e),(g^+_e)' )$ with~$(B(x_e^-,R_e)-B(x_e^-,r_e),(g^-_e)' )$ by identifying their inner boundaries, both of which are isometric to the round sphere $\Sp^2(r_e)$ of radius $r_e$. The resulting Riemannian manifold $(M',g')$ is homeomorphic to $M$ by construction, and by Proposition \ref{pr:surgery_metric}, the Riemannian metric $g'$ has uniformly positive scalar curvature $\scal_{g'} \geq \alpha s_0 > 0$.
\end{proof}

\begin{remark}
	The same strategy was used in \cite{Bessieres_Besson_Maillot_Marques_2021}, where the authors proved that an orientable 3-manifold admits a complete Riemannian metric of positive scalar curvature and bounded geometry if and only if the manifold decomposes as a possibly infinite connected sum of spherical 3-manifolds and $\Sp^2 \times \Sp^1$ with finitely many summands up to homeomorphism.
\end{remark}

\section{Statements and Declarations}

The authors have no competing interests to declare that are relevant to the content of this article. Data sharing not applicable to this article as no datasets were generated or analysed during the current study.

\pagestyle{plain}
\bibliographystyle{alpha}
\bibliography{REF_Decomposition3Manifolds}

\begin{thebibliography}{BBMM21}

\bibitem[BBM11]{Bessieres_Besson_Maillot_2011}
Laurent Bessi\`eres, G\'{e}rard Besson, and Sylvain Maillot.
\newblock Ricci flow on open 3-manifolds and positive scalar curvature.
\newblock {\em Geom. Topol.}, 15(2):927--975, 2011.

\bibitem[BBMM21]{Bessieres_Besson_Maillot_Marques_2021}
Laurent Bessi\`eres, G\'erard Besson, Sylvain Maillot, and Fernando~C. Marques.
\newblock Deforming 3-manifolds of bounded geometry and uniformly positive
  scalar curvature.
\newblock {\em J. Eur. Math. Soc. (JEMS)}, 23(1):153--184, 2021.

\bibitem[CL24]{Chodosh_Li_2024}
Otis Chodosh and Chao Li.
\newblock Generalized soap bubbles and the topology of manifolds with positive
  scalar curvature.
\newblock {\em Ann. of Math. (2)}, 199(2):707--740, 2024.

\bibitem[CLL23]{CLL_2023}
Otis Chodosh, Chao Li, and Yevgeny Liokumovich.
\newblock Classifying sufficiently connected {PSC} manifolds in 4 and 5
  dimensions.
\newblock {\em Geom. Topol.}, 27(4):1635--1655, 2023.

\bibitem[CLX25]{Chodosh_Lai_Xu_2025}
Otis Chodosh, Yi~Lai, and Kai Xu.
\newblock 3-manifolds with positive scalar curvature and bounded geometry.
\newblock arXiv:2502.09727, 2025.

\bibitem[CWY10]{Chang_Weinberger_Yu_2010}
Stanley Chang, Shmuel Weinberger, and Guoliang Yu.
\newblock Taming 3-manifolds using scalar curvature.
\newblock {\em Geom. Dedicata}, 148:3--14, 2010.

\bibitem[Dun85]{Dunwoody_1985}
Martin~J. Dunwoody.
\newblock The accessibility of finitely presented groups.
\newblock {\em Invent. Math.}, 81(3):449--457, 1985.

\bibitem[Edw63]{Edwards_1963}
C.~H. Edwards, Jr.
\newblock Open {$3$}-manifolds which are simply connected at infinity.
\newblock {\em Proc. Amer. Math. Soc.}, 14:391--395, 1963.

\bibitem[Eps61]{Epstein_1961}
David B.~A. Epstein.
\newblock Ends.
\newblock In {\em Topology of 3-manifolds and related topics ({P}roc. {T}he
  {U}niv. of {G}eorgia {I}nstitute, 1961)}, pages 110--117. Prentice-Hall,
  Inc., Englewood Cliffs, NJ, 1961.

\bibitem[Fre82]{Freedman_1982}
Michael~Hartley Freedman.
\newblock The topology of four-dimensional manifolds.
\newblock {\em J. Differential Geometry}, 17(3):357--453, 1982.

\bibitem[Geo08]{Geoghegan_2008}
Ross Geoghegan.
\newblock {\em Topological methods in group theory}, volume 243 of {\em
  Graduate Texts in Mathematics}.
\newblock Springer, New York, 2008.

\bibitem[GHL04]{GHL_2004}
Sylvestre Gallot, Dominique Hulin, and Jacques Lafontaine.
\newblock {\em Riemannian geometry}.
\newblock Universitext. Springer-Verlag, Berlin, third edition, 2004.

\bibitem[GL80a]{Gromov_Lawson_1980a}
Mikhael Gromov and H.~Blaine Lawson, Jr.
\newblock {\GG{a}}{S}pin and scalar curvature in the presence of a fundamental
  group. {I}.
\newblock {\em Ann. of Math. (2)}, 111(2):209--230, 1980.

\bibitem[GL80b]{Gromov_Lawson_1980b}
Mikhael Gromov and H.~Blaine Lawson, Jr.
\newblock {\GG{b}}{T}he classification of simply connected manifolds of
  positive scalar curvature.
\newblock {\em Ann. of Math. (2)}, 111(3):423--434, 1980.

\bibitem[GL83]{Gromov_Lawson_1983}
Mikhael Gromov and H.~Blaine Lawson, Jr.
\newblock Positive scalar curvature and the {D}irac operator on complete
  {R}iemannian manifolds.
\newblock {\em Inst. Hautes \'{E}tudes Sci. Publ. Math.}, (58):83--196, 1983.

\bibitem[Gro83]{Gromov_1983}
Mikhael Gromov.
\newblock Filling {R}iemannian manifolds.
\newblock {\em J. Differential Geom.}, 18(1):1--147, 1983.

\bibitem[Gro86]{Gromov_1986}
Mikhael Gromov.
\newblock Large {R}iemannian manifolds.
\newblock In {\em Curvature and topology of {R}iemannian manifolds ({K}atata,
  1985)}, volume 1201 of {\em Lecture Notes in Math.}, pages 108--121.
  Springer, Berlin, 1986.

\bibitem[Gro17]{Gromov_2017}
Mikhael Gromov.
\newblock 101 questions, problems and conjectures around scalar curvature.
  (incomplete and unedited version).
\newblock 2017.

\bibitem[Gro20]{Gromov_2020}
Mikhael Gromov.
\newblock No metrics with positive scalar curvatures on aspherical 5-manifolds.
\newblock arXiv:2009.05332, 2020.

\bibitem[Gro23]{Gromov_2023}
Mikhael Gromov.
\newblock Four lectures on scalar curvature.
\newblock In {\em Perspectives in scalar curvature. {V}ol. 1}, pages 1--514.
  World Sci. Publ., Hackensack, NJ, 2023.

\bibitem[Hat02]{Hatcher_2002}
Allen Hatcher.
\newblock {\em Algebraic topology}.
\newblock Cambridge University Press, Cambridge, 2002.

\bibitem[Hat04]{Hatcher_2004}
Allen Hatcher.
\newblock The classification of 3-manifolds -- {A} brief overview.
\newblock 2004.
\newblock Available at
  \url{https://pi.math.cornell.edu/~hatcher/Papers/3Msurvey.pdf}.

\bibitem[Hem76]{Hempel_1976}
John Hempel.
\newblock {\em {$3$}-{M}anifolds}, volume No. 86 of {\em Annals of Mathematics
  Studies}.
\newblock Princeton University Press, Princeton, NJ; University of Tokyo Press,
  Tokyo, 1976.

\bibitem[HXZ24]{Hu_Xu_Zhang_2024}
Qixuan Hu, Guoyi Xu, and Shuai Zhang.
\newblock The sharp diameter bound of stable minimal surfaces.
\newblock arXiv:2411.18928v3, 2024.

\bibitem[Kne29]{Kneser_1929}
Hellmuth Kneser.
\newblock Geschlossene flächen in dreidimensionalen mannigfaltigkeiten.
\newblock {\em Jahresbericht der Deutschen Mathematiker-Vereinigung},
  38:248--259, 1929.

\bibitem[Mai08]{Maillot_2008}
Sylvain Maillot.
\newblock Some open 3-manifolds and 3-orbifolds without locally finite
  canonical decompositions.
\newblock {\em Algebr. Geom. Topol.}, 8(3):1795--1810, 2008.

\bibitem[Mil62]{Milnor_1962}
John Milnor.
\newblock A unique decomposition theorem for {$3$}-manifolds.
\newblock {\em Amer. J. Math.}, 84:1--7, 1962.

\bibitem[MY82]{Meeks_Yau_1982}
William~H. Meeks, III and Shing-Tung Yau.
\newblock The classical {P}lateau problem and the topology of three-dimensional
  manifolds. {T}he embedding of the solution given by {D}ouglas-{M}orrey and an
  analytic proof of {D}ehn's lemma.
\newblock {\em Topology}, 21(4):409--442, 1982.

\bibitem[Pap57]{Papakyriakopoulos_1957}
Christos~D. Papakyriakopoulos.
\newblock On {D}ehn's lemma and the asphericity of knots.
\newblock {\em Proc. Nat. Acad. Sci. U.S.A.}, 43:169--172, 1957.

\bibitem[Per02]{Perelman_2002}
Grisha Perelman.
\newblock The entropy formula for the {R}icci flow and its geometric
  applications.
\newblock arXiv:math/0211159, 2002.

\bibitem[Per03a]{Perelman_2003_a}
Grisha Perelman.
\newblock Finite extinction time for the solutions to the {R}icci flow on
  certain three-manifolds.
\newblock arXiv:math/0307245, 2003.

\bibitem[Per03b]{Perelman_2003_b}
Grisha Perelman.
\newblock Ricci flow with surgery on three-manifolds.
\newblock arXiv:math/0303109, 2003.

\bibitem[RW10]{Ramachandran_Wolfson_2010}
Mohan Ramachandran and Jon Wolfson.
\newblock Fill radius and the fundamental group.
\newblock {\em J. Topol. Anal.}, 2(1):99--107, 2010.

\bibitem[Sco77]{Scott_1977}
Peter Scott.
\newblock Fundamental groups of non-compact {$3$}-manifolds.
\newblock {\em Proc. London Math. Soc. (3)}, 34(2):303--326, 1977.

\bibitem[Ser80]{Serre_1980}
Jean-Pierre Serre.
\newblock {\em Trees}.
\newblock Springer-Verlag, Berlin-New York, 1980.
\newblock Translated from the French by John Stillwell.

\bibitem[ST89]{Scott_Tucker_1989}
Peter Scott and Thomas Tucker.
\newblock Some examples of exotic noncompact {$3$}-manifolds.
\newblock {\em Quart. J. Math. Oxford Ser. (2)}, 40(160):481--499, 1989.

\bibitem[Sta62]{Stallings_1962}
John Stallings.
\newblock The piecewise-linear structure of {E}uclidean space.
\newblock {\em Proc. Cambridge Philos. Soc.}, 58:481--488, 1962.

\bibitem[SY79a]{Schoen_Yau_1979_a}
Richard Schoen and Shing-Tung Yau.
\newblock Existence of incompressible minimal surfaces and the topology of
  three-dimensional manifolds with nonnegative scalar curvature.
\newblock {\em Ann. of Math. (2)}, 110(1):127--142, 1979.

\bibitem[SY79b]{Schoen_Yau_1979_b}
Richard Schoen and Shing-Tung Yau.
\newblock On the structure of manifolds with positive scalar curvature.
\newblock {\em Manuscripta Math.}, 28(1-3):159--183, 1979.

\bibitem[SY83]{Schoen_Yau_1983}
Richard Schoen and Shing-Tung Yau.
\newblock The existence of a black hole due to condensation of matter.
\newblock {\em Comm. Math. Phys.}, 90(4):575--579, 1983.

\bibitem[Wan19]{Wang_2019}
Jian Wang.
\newblock Simply connected open 3-manifolds with slow decay of positive scalar
  curvature.
\newblock {\em Comptes Rendus Mathematique}, 357(3):284--290, 2019.

\bibitem[Wan23a]{Wang_2023}
Jian Wang.
\newblock Topology of $3$-manifolds with uniformly positive scalar curvature.
\newblock arXiv:2212.14383v2, 2023.

\bibitem[Wan23b]{Wang_2023_Review}
Jian Wang.
\newblock Topological characterization of contractible 3-manifolds with
  positive scalar curvature.
\newblock In {\em Perspectives in scalar curvature. {V}ol. 2}, pages 313--321.
  World Sci. Publ., Hackensack, NJ, [2023] \copyright2023.

\bibitem[Wan24a]{Wang_2024_JDG}
Jian Wang.
\newblock {Contractible 3-manifold and Positive scalar curvature (I)}.
\newblock {\em Journal of Differential Geometry}, 127(3):1267 -- 1304, 2024.

\bibitem[Wan24b]{Wang_2024_JEMS}
Jian Wang.
\newblock Contractible 3-manifolds and positive scalar curvature ({I}{I}).
\newblock {\em J. Eur. Math. Soc.}, 26(2):537--572, 2024.

\bibitem[Wol12]{Wolfson_2012}
Jon Wolfson.
\newblock Manifolds with {$k$}-positive {R}icci curvature.
\newblock In {\em Variational problems in differential geometry}, volume 394 of
  {\em London Math. Soc. Lecture Note Ser.}, pages 182--201. Cambridge Univ.
  Press, Cambridge, 2012.

\bibitem[Yau82]{Yau_1982}
Shing-Tung Yau.
\newblock Problem section.
\newblock In {\em Seminar on {D}ifferential {G}eometry}, volume No. 102 of {\em
  Ann. of Math. Stud.}, pages 669--706. Princeton Univ. Press, Princeton, NJ,
  1982.

\end{thebibliography}

\end{document}